\RequirePackage{etex}
\RequirePackage{easymat}
\documentclass[12pt]{elsarticle}
\usepackage{amsmath,amsthm,amssymb,extsizes}

\usepackage[all]{xy}
\usepackage[active]{srcltx}
\newcommand{\F}{\mathbb F}
\newcommand{\T}{\mathbb T}
\newcommand{\V}{\mathbb V}

\newcommand{\mc}{\mathcal}

\newcommand{\ci}{
\begin{picture}(6,6)
\put(3,3
){\circle{3}}
\end{picture}}

\newcommand{\cim}{
\begin{picture}(6,6)
\put(3,3){\circle*{3}}
\end{picture}}

\newcommand{\lin}{\xymatrix@C=17pt@1{{}\ar@{-}[r]&{}}}

\DeclareMathOperator{\rank}{rank}
\DeclareMathOperator{\diag}{diag}
\renewcommand{\le}{\leqslant}
\renewcommand{\ge}{\geqslant}
\renewcommand{\u}{\underline}
\renewcommand{\v}{\underrightarrow}
\renewcommand{\o}{\bar}
\newcommand{\too}%
{\xrightarrow{\text{\raisebox{-3pt}{$\sim$}}\,}}

\newtheorem{theorem}{Theorem}[section]
\newtheorem{lemma}{Lemma}[section]
\newtheorem{corollary}{Corollary}[section]

\theoremstyle{definition}
\newtheorem{definition}{Definition}[section]
\newtheorem{example}{Example}[section]

\theoremstyle{remark}
\newtheorem{remark}{Remark}[section]

\begin{document}
\title{Wildness for tensors}

\author[fut]{Vyacheslav Futorny}
\ead{futorny@ime.usp.br}

\address[fut]{Department of Mathematics, University of S\~ao Paulo, Brazil}

\author[gro]{Joshua A. Grochow}
\ead{jgrochow@colorado.edu}

\address[gro]{Departments of Computer Science and Mathematics,
University of Colorado Boulder,
Boulder, CO, USA}

\author[ser]{Vladimir V.
Sergeichuk\corref{cor}}
\ead{sergeich@imath.kiev.ua}
\address[ser]{Institute of Mathematics,
Kiev, Ukraine}
\cortext[cor]{Corresponding author.}

\begin{abstract}
In representation theory, a classification problem is called wild if it contains the problem of classifying matrix pairs up to simultaneous similarity. The latter problem is considered hopeless; it contains the problem of classifying an arbitrary finite system of vector spaces and linear mappings between them. We prove that an analogous ``universal'' problem in the theory of tensors of order at most 3 over an arbitrary field is the problem of classifying three-dimensional arrays up to
equivalence transformations
\[
[a_{ijk}]_{i=1}^{m}\,{}_{j=1}^{n}\,{}_{k=1}^{t}\ \mapsto\ \Bigl[
\sum_{i,j,k}
a_{ijk}u_{ii'}
v_{jj'}w_{kk'}\Bigr]{}_{i'=1}^{m}\,{}_{j'=1}^{n}\,{}_{k'=1}^{t}
\]
in which $[u_{ii'}]$, $[v_{jj'}]$, $[w_{kk'}]$ are nonsingular matrices: this problem contains the problem of classifying
an arbitrary system of tensors of order at
most three.

\end{abstract}


\maketitle

\section{Introduction and main result}
\label{ss1}

We prove that
\begin{equation}\label{zzzz}
\parbox[c]{0.8\textwidth}{the problem of classifying
three-dimensional arrays up to
equivalence transformations
\[
[a_{ijk}]_{i=1}^{m}\,{}_{j=1}^{n}\,{}_{k=1}^{t}\ \mapsto\ \Bigl[
\sum_{i,j,k}
a_{ijk}u_{ii'}
v_{jj'}w_{kk'}\Bigr]{}_{i'=1}^{m}\,{}_{j'=1}^{n}\,{}_{k'=1}^{t}
\]
in which $[u_{ii'}]$, $[v_{jj'}]$, $[w_{kk'}]$ are nonsingular matrices}
\end{equation}
``contains''
the problem of classifying an
arbitrary system of tensors of order at
most three; which means that the solution of the second problem can be derived from the solution of the first (see Definition \ref{cbd} of the notion ``contains'').

In some precise, the problem of classifying matrix pairs up to simultaneous similarity contains all classification problems for systems of linear mappings (see Section \ref{sect}). We show that \eqref{zzzz} is an analogous universal problem for systems of tensors of order at most three.

We are essentially concerned with systems of arrays. However, the main theorem is formulated in  Section \ref{ssw} in terms of systems of tensors that are considered
as \emph{representations  of directed bipartite graphs} (i.e., directed graphs in which the set of
vertices is partitioned into two
subsets and all the arrows are between
these subsets).
The vertices $t_1,\dots,t_p$ on the left represent tensors and the vertices $1,\dots,q$ on the right represent vector spaces. 

For
example, a representation of the graph
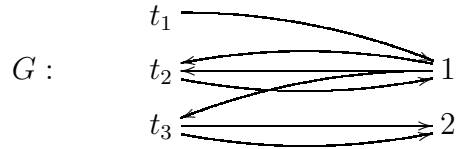
\begin{equation*}
\begin{split}
\hspace{-3.5cm}{\xymatrix%
@C=95pt@R=5pt{
 &{t_1}\ar@/^/@<0.3ex>[dr]&
  \\{G:}\hspace{-4cm} &{t_2}\ar@/_/@<-0.3ex>[r]
  &{1}
  \ar@/_/@<-0.3ex>[l]
  \ar[l]\ar@/_/[ld]
 \\&{t_3}\ar@/_/@<-0.3ex>[r]
 &{2}\ar@{<-}[l]
  }}
  \end{split}
\end{equation*}
is a system
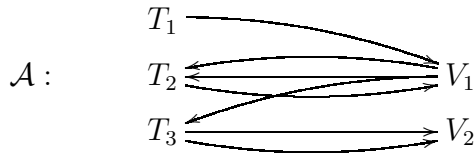
\begin{equation*}
\begin{split}
\hspace{-3.5cm}{\xymatrix@C=95pt@R=5pt{
 &{T_1}\ar@/^/@<0.3ex>[dr]&
  \\{\mc A:}\hspace{-4cm}&{T_2}\ar@/_/@<-0.3ex>[r]
  &{V_1}
  \ar@/_/@<-0.3ex>[l]
  \ar[l]\ar@/_/[ld]
 \\&{T_3}\ar@/_/@<-0.3ex>[r]
 &{V_2}\ar@{<-}[l]
  }}
  \end{split}
\end{equation*}
of vector spaces $V_1$ and $V_2$ over a field $\F$ and tensors
\[
T_1\in V_1^*,\quad T_2\in V_1\otimes V_1
\otimes V_1^*,\quad T_3\in V_1\otimes V_2^*\otimes V_2^*,
\]
in which $V^*$ denotes the dual space of $V$.

The \emph{dimension} of a representation is the vector $(\dim V_1,\dots,\dim V_q)$.
Two representations $\mc A$ and $\mc A'$ of the same dimension are \emph{isomorphic} if $\mc A$ is transformed to $\mc A'$ by a system of linear bijections $V_1\to V_1',\dots,V_q\to V_q'$.
All representations of dimension $\underline n:=(n_1,\dots,n_q)$ are isomorphic to representations with
$V_1=\mathbb F^{n_1},\dots,V_q=\mathbb F^{n_q}$, whose sequences $(T_1,\dots,T_p)$ of tensors  are given by sequences $A=(A_1,\dots,A_p)$ of arrays over $\mathbb F$.
These sequences of arrays form the vector space, which we denote by $\mathcal W_{\underline n}(G)$.

Let $G$ be a directed bipartite graph, in which each vertex has  at most  $3$ arrows (and so each representation of $G$ consists of tensors of order  at most  $3$). The aim of this paper is to construct
an affine injection $F: \mc W_{\underline n}(G)\to\F^{m_1\times m_2\times m_3}$ with the following property:
\emph{two representations $A,A'\in\mc W_{\underline n}(G)$ are isomorphic if and only if the array $F(A)$ is transformed to $F(A')$ by equivalence transformations \eqref{zzzz}} (see Theorem \ref{ktw1}).

Note that the problem of classifying tensors of order $3$ is motivated from seemingly independent questions in mathematics, physics, and computational complexity. Each finite dimensional algebra is given by a $(1,2)$-tensor; see Example \ref{mkw}. In computer science, this problem plays a role in algorithms for testing isomorphism of finite groups \cite{BMW}, algorithms for testing polynomial identities  \cite{AGLOW,IQS1, IQS2}, and understanding the boundary of the determinant orbit closure \cite{LMR, grochowPhD}. This problem also arises in the classification of quantum entangled states, which has applications in physics and quantum computing  \cite{jozsaLinden,kendonMunro,WGE}.

All arrays and tensors that we consider are over an arbitrary field $\mathbb F$.

\subsection{Wild problems}\label{sect}

Our paper was motivated by the theory of wild matrix problems; in this section we recall some known facts.

A classification problem over a field $\mathbb F$  is called
\emph{wild} if it contains
\begin{equation}\label{yhn}
\parbox[c]{0.8\textwidth}{the problem
of classifying pairs $(A,B)$ of square matrices of the same size over $\mathbb F$ up to transformations of
simultaneous similarity
$(S^{-1}AS,S^{-1}BS)$, in which
$S$ is a nonsingular matrix;}
\end{equation}
see formal definitions in  \cite{dro}, \cite{gnrs}, and \cite[Section 14.10]{gab-roi}.

Gelfand and Ponomarev \cite{gel-pon}
proved that  the problem \eqref{yhn} (and even the problem  of classifying
pairs $(A,B)$ of
commuting nilpotent matrices up to
simultaneous similarity) contains the
problem of classifying $t$-tuples of square matrices of the same size up to transformations of simultaneous similarity
\[
(M_1,\dots,M_t)\mapsto (C^{-1}M_1C,\dots,C^{-1}M_tC),\qquad
C\text{ is nonsingular}.
\]

\begin{example}\label{fsq}
Gelfand and Ponomarev's statement, but without the condition of commutativity of $A$ and $B$, is easily proved.
For each $t$-tuple $\mc M=(M_1,\dots,M_t)$ of $m\times m$ matrices, we define two  $(t+1)m\times(t+1)m$ nilpotent matrices
\[
A:=\begin{bmatrix}
  0 & I_m&&0 \\ &0&\ddots\\&&\ddots&I_m\\
  0 & &&0 \\
\end{bmatrix},\qquad
B(\mc M):=\begin{bmatrix}
  0 & M_1&&0 \\ &0&\ddots\\&&\ddots&M_t\\
  0 & &&0 \\
\end{bmatrix}.
\]
Let $\mc N=(N_1,\dots,N_t)$ be another $t$-tuple of $m\times m$ matrices.
Then $\mc M$ and $\mc N$ are similar if and only if $(A,B(\mc M))$ and $(A,B(\mc N))$ are similar. Indeed, let  $(A,B(\mc M))$ and $(A,B(\mc N))$ be similar; that is,
\begin{equation}\label{6r6}
AS=SA,\qquad B(\mc M)S=SB(\mc N)
\end{equation}
for a nonsingular $S$.  The first equality in \eqref{6r6} implies that $S$ has an upper block-triangular form
\[
S=\begin{bmatrix}
  C &&&* \\ &C\\&&\ddots\\
  0 & &&C \\
\end{bmatrix},\qquad C\text{ is }m\times m.
\]
Then the second equality  in \eqref{6r6} implies that $
M_1C=CN_1,\dots, M_tC=CN_t$. Therefore, $\mc M$ is similar to $\mc N$. Conversely, if $\mc M$ is similar to $\mc N$ via $C$, then $(A,B(\mc M))$ is similar to $(A,B(\mc N))$ via $\diag(C,\dots,C)$.
\end{example}

\begin{example}\label{fiq}
The problem of classifying pairs $(M,N)$ of $m\times n$ and $m\times m$ matrices up to transformations  \begin{equation}\label{nf}
(M,N)\mapsto (C^{-1}MR,C^{-1}NC),\qquad
C\text{ and $R$ are nonsingular,}
\end{equation}
looks simpler than the problem of classifying matrix pairs up to similarity since \eqref{nf} has additional admissible transformations. However, these problems have the same complexity since for each two pairs $(A,B)$ and $(A',B')$ of $n\times n$ matrices the pair
\[
\left(
\begin{bmatrix}
  I_n\\0_n \\
\end{bmatrix},
\begin{bmatrix}
  0_n &A\\I_n&B \\
\end{bmatrix}
\right)\text{ is reduced to }
\left(
\begin{bmatrix}
  I_n\\0_n \\
\end{bmatrix},
\begin{bmatrix}
  0_n &A'\\I_n&B' \\
\end{bmatrix}
\right)
\]
by transformations \eqref{nf}  if and only if $(A,B)$ is similar to $(A',B')$.
\end{example}

Moreover, by
\cite{bel-ser_compl} the problem \eqref{yhn} contains the
problem of classifying representations
of an arbitrary quiver  over a field $\mathbb F$ (i.e., of an
arbitrary finite set of vector spaces  over $\mathbb F$
and linear mappings between them) and  the
problem of classifying
representations of an arbitrary
partially ordered set. Analogously, by \cite{debora}
the problem of classifying  pairs $(A,B)$ of commuting complex matrices of the same size up to transformations of simultaneous consimilarity
$(\bar S^{-1}AS,\bar S^{-1}BS)$, in which
$S$ is nonsingular, contains the problem of classifying an
arbitrary finite set of complex vector spaces
and linear or semilinear mappings between them.

Thus, all wild classification problems for
systems of linear mappings have the
same complexity and a solution of any one
of them would imply a solution of every
wild or non-wild problem.

The universal role of the problem \eqref{yhn} is not extended to systems of tensors: Belitskii and Sergeichuk
\cite{bel-ser_compl} proved that the problem \eqref{yhn} is contained
in the problem of classifying
three-dimensional arrays up to
equivalence but does not contain it.

\subsection{Organization of the paper}

The main theorem is formulated in Section \ref{ssw}. Its proof is given in Sections \ref{ss2}--\ref{sss4}, in which we successively  prove special cases of the main theorem. We describe them in this section.

\begin{definition}\label{lke}
 An \emph{array of size
$d_1\times\dots\times d_r$} over a
field $\F$ is an indexed
collection $\u A=[a_{i_1\dots
i_r}]_{i_1=1}^{d_1}{}\dots{}_{i_r=1}^{d_r}$
of elements of $\F$. (We denote arrays by underlined capital letters.)
Let $\u A=[a_{i_1\dots
i_r}]$ and $\u
B=[b_{i_1\dots i_r}]$ be two arrays of size
$d_1\times\dots\times d_r$ over a
field $\F$. If there exist
nonsingular matrices
$S_1=[s_{1ij}]\in\F^{d_1\times d_1}$, \dots, $S_r=[s_{rij}]\in\F^{d_r\times d_r}$
such that
\begin{equation}\label{keje}
b_{j_1\dots j_r}=\sum_{i_1,\dots, i_r}a_{i_1\dots i_r}s_{1i_1j_1}
\dots s_{ri_rj_r}
\end{equation}
for all ${j_1,\dots ,j_r}$,
then we say that $\u A$ and $\u B$ are
\emph{equivalent} and write
\begin{equation}\label{mkt}
(S_1,\dots,S_r): \u A \too \u
B.
\end{equation}
\end{definition}

We define partitioned three-dimensional arrays by analogy with block matrices as follows.
Let $\u A=[a_{ijk}]_{i=1}
^{m}\,{}_{j=1}^{n}\,{}_{k=1}^{t}
$ be an array of size $m\times n\times t$.
Each partition of its index sets
\begin{equation}\label{1jkk}
 \left.
 \begin{aligned}
\{1,\dots,m\}&=\{1,\dots,i_1\,|
\,i_1+1,\dots,i_2\,
|\,\dots\,|
\,i_{\o m-1}+1,\dots,m\}
        \\
\{1,\dots,n\}&=\{1,\dots,j_1\,|
\,j_1+1,\dots,j_2\,
|\,\dots\,|\,
j_{\o n-1}+1,\dots,n\}
           \\
\{1,\dots,t\}&=\{1,\dots,k_1\,|
\,k_1+1,\dots,k_2\,
|\,\dots\,|
\,k_{\o t -1}+1,\dots,t\}
\end{aligned}
\right\}
\end{equation}
(we set $i_0=j_0=k_0:=0$, $i_{\o m}:=m$, $j_{\o n}:=n$, and $k_{\o t}:=t$) defines the \emph{partitioned array}
\begin{equation*}\label{vfj} \u A=[\u
A_{\alpha \beta\gamma}]_{\alpha
=1}^{\o m}\,{}_{\beta=1}^{\o
n}\,{}_{\gamma=1}^{\o t}
\end{equation*}
\emph{with $\o m\cdot \o
n\cdot\o t$ spatial blocks}
\[
\u A_{\alpha \beta\gamma}:=[a_{ijk}]
_{i=i_{\alpha-1}+1}
^{i_{\alpha}}{}\
_{j=j_{\beta-1}+1}
^{j_{\beta}}{}\
_{k=k_{\gamma-1}+1}
^{k_{\gamma}}.
\]
Thus, $\u A$ is
partitioned into spatial blocks $\u
A_{\alpha \beta \gamma }$ by
frontal, lateral, and horizontal
planes.

Two partitioned arrays
\begin{equation}\label{dsj}
\u A=[\u
A_{\alpha \beta\gamma}]_{\alpha
=1}^{\o m}\,{}_{\beta=1}^{\o
n}\,{}_{\gamma=1}^{\o t}
         \quad\text{and}\quad
\u B=[\u
B_{\alpha \beta\gamma}]_{\alpha
=1}^{\o m}\,{}_{\beta=1}^{\o
n}\,{}_{\gamma=1}^{\o t}
\end{equation}
of the same size are \emph{conformally partitioned} if the sizes of the space blocks $\u
A_{\alpha \beta\gamma}$ and $\u
B_{\alpha \beta\gamma}$ are equal for each $\alpha, \beta,\gamma$.

Two conformally partitioned three-dimensional arrays
\eqref{dsj} whose partition is given by \eqref{1jkk}, are \emph{block-equivalent}
if there exists an equivalence
$(S_1,S_2,S_3): \u A \too \u
B$ (see
 \eqref{mkt}) in which
\begin{equation}\label{nmr}
S=S_{11}\oplus \dots\oplus S_{1\o m},\quad
S_2=S_{21}\oplus \dots\oplus S_{2\o n},
 \quad
S_3=S_{31}\oplus \dots\oplus S_{3\o t}
\end{equation}
and the sizes of diagonal blocks in \eqref{nmr} are given by \eqref{1jkk}.

In Section \ref{ss2} we prove Theorem \ref{hep}, which implies that
\begin{equation}\label{kjt1}
\parbox[c]{0.8\textwidth}
{the problem \eqref{zzzz} contains the problem of classifying partitioned
three-dimensional arrays up to
block-equivalence.}
\end{equation}
Theorem \ref{hep} is our main tool in the proof of Theorem \ref{ktw1}.

In Section \ref{sss3}, we prove Corollary \ref{dek12}, which implies that
\begin{equation}\label{kjt2}
\parbox[c]{0.8\textwidth}
{for an arbitrary $t$, the problem \eqref{zzzz} contains the problem of classifying $t$-tuples of three-dimensional arrays up to
simultaneous equivalence,}
\end{equation}
which is a three-dimensional analogue of Gelfand and Ponomarev's statement from  \cite{gel-pon} about the problem \eqref{yhn}.

In Section \ref{sss4}, we consider \emph{linked block-equivalence transformations} of three-dimensional arrays; that is, block-equivalence transformations $(S_1,S_2,S_3): \u A \too \u
B$ of the form \eqref{nmr}, in which some of the diagonal blocks are claimed to be equal ($S_{ij}=S_{i'j'}$) and some of the diagonal blocks are claimed to be mutually contragredient ($S_{ij}=(S_{i'j'}^{-1})^T$). We prove Theorem \ref{kts}, which implies that
\begin{equation}\label{kjj}
\parbox[c]{0.8\textwidth}
{the problem \eqref{zzzz} contains the problem of classifying partitioned
three-dimensional arrays up to linked
block-equivalence.
}
\end{equation}

The main result of the article is Theorem \ref{ktw1}, which generalizes \eqref{kjt1}--\eqref{kjj} and means that
\begin{equation}\label{kjt3}
\parbox[c]{0.8\textwidth}
{\emph{the problem \eqref{zzzz} contains  the problem of classifying an arbitrary system of tensors of order at
most 3.
}}
\end{equation}

Note that the second problem in \eqref{kjt3} contains both the problem of classifying systems of linear mappings and bilinear forms (i.e., representations of mixed graphs)   and the problem of classifying finite dimensional algebras; see \cite{hor-ser,ser_izv} and Example \ref{err}.

    \begin{remark}
    Because of the potential applications in computational complexity, we remark that all of the containments we construct are easily seen to be uniform $p$-projections in the sense of Valiant \cite{valiant}. In this way, our containments not only show that mathematically one classification problem contains another, but also that this holds in an effective, computational sense. In particular, a polynomial-time algorithm for testing equivalence of three-dimensional arrays would yield a polynomial-time algorithm for all the other problems considered in this paper. (Perhaps the only caveat to be aware of is that for partitioned arrays with $t$ parts, the reduction is polynomial in the size of the array and $2^t$.)
    \end{remark}

\subsection{Main theorem}
\label{ssw}

All systems of tensors of fixed orders $\le 3$ form a vector space.  We construct in Theorem \ref{ktw1} an embedding of this vector space into the vector space of three-dimensional arrays of a fixed size (the image of this embedding is an affine space) is such a way that \emph{two systems of tensors are isomorphic if and only if their images are equivalent arrays}; see Remark \ref{rema}.

An array $\u A=[a_{i_1\dots
i_r}]_{i_1=1}^{d_1}{}\dots{}_{i_r=1}^{d_r}$ is a \emph{subarray} of an array $\u B=[b_{j_1\dots
j_{\rho}}]_{j_1=1}^{{\delta}_1}{}
\dots{}_{j_{\rho}=1}^{{\delta}_{\rho}}$ if $r\le \rho $ and there are nonempty (possible, one-element) subsets $$J_1\subset\{1,\dots,{\delta}_1\},\ \dots,\ J_{\rho }\subset\{1,\dots,{\delta}_{\rho }\}$$ such that $\u A$ coincides with $[b_{j_1\dots
j_{\rho}}]_{j_1\in J_1
,\dots, j_{\rho}\in J_{\rho}}$ up to deleting the indices $j_k$ with one-element $J_k$. The \emph{size of a $p$-tuple $\mc A=(\u A_1,\dots,\u A_p)$ of arrays}  is the sequence $d:=(\u d_1,\dots,\u d_p)$ of their sizes. Each array classification problem $\mc C_d$ that we consider is given by a set of $\mc C_d$-admissible transformations on the set of array $p$-tuples of size $d$.

We use the following definition of embedding of one classification problem about systems of aggregates to another, which generalizes the constructions from Examples \ref{fsq} and \ref{fiq}. This definition is general; its concrete realization is given in Theorem \ref{ktw1}.

 \begin{definition}\label{cbd}
Let ${\cal X}=(\u X_1,\dots,\u X_p)$ be a variable array $p$-tuple of size $d$, in which the entries of $\u X_1,\dots,\u X_p$ are independent variables (without repetition). We say that \emph{an array classification problem $\mc C_d$ is contained in an array classification problem $\mc D_{\delta }$} if there is a variable array $\pi $-tuple $\mc F(\mc X)=(\u F_1,\dots, \u F_{\pi})$ of size $\delta $, in which every $\u X_i$ is a subarray of some $\u F_j$ and each entry of $\u F_1,\dots, \u F_{\pi}$ outside of $\u X_1,\dots,\u X_{p}$ is $0$ or $1$, such that
\begin{quote}
    $\mc A$ is reduced to $\mc B$ by $\mc C_d$-admissible transformations if and only if
$\mc F(\mc A)$ is reduced to $\mc F(\mc B)$ by $\mc D_{\delta }$-admissible transformations
\end{quote}
 for all array $p$-tuples  $\mc A$ and $\mc B$ of size $d$.
 \end{definition}

Note that $\mc F(\mc X)$ defines an affine map $\mc A\mapsto \mc F(\mc A)$ of the vector space of all array $p$-tuples of size $d$
to the vector space of all array $\pi $-tuples of size $\delta $.

Gabriel \cite{gab} (see also
\cite{gab-roi,hor-ser}) suggested to
consider systems of vector spaces and
linear mappings as representations of
quivers: a quiver is a directed graph;
its representation is given by
assigning a vector space to each vertex
and a linear mapping of the corresponding vector spaces to each arrow.
Generalizing this notion, Sergeichuk
\cite{ser_sur} suggested to study
systems of tensors as
representations of directed bipartite
graphs. These representations are
another form of \emph{Penrose's tensor
diagrams} \cite{pen}, which are studied
in \cite{cvi,tur}.

A \emph{directed bipartite graph $G$}
is a directed graph in which the set of
vertices is partitioned into two
subsets and all the arrows are between
these subsets. We denote these subsets
by ${\T}$ and $\V$, and write the
vertices from $\T$ on the left and the
vertices from $\V$ on the right. For
example,
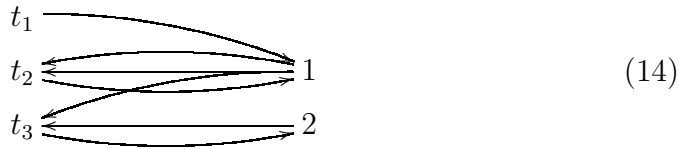
\begin{equation}\label{hle1}
\begin{split}
{\xymatrix@C=95pt@R=5pt{
 {t_1}\ar@/^/@<0.3ex>[dr]&
  \\{t_2}\ar@/_/@<-0.3ex>[r]
  &{1}
  \ar@/_/@<-0.3ex>[l]
  \ar[l]\ar@/_/[ld]
 \\{t_3}\ar@/_/@<-0.3ex>[r]
 &{2}\ar[l]
  }}
  \end{split}
\end{equation}
is a directed bipartite graph,
in which $\T=\{t_1,t_2,t_3\}$ and $\V=\{1,2\}$.

The following definition of
representations of directed bipartite
graphs is given in terms of arrays. We show in Section \ref{hyu} that it is
equivalent to the
definition from \cite{ser_sur} given in terms of tensors that are considered as elements of tensor products.

\begin{definition}\label{aal}
Let $G$ be a directed bipartite
graph with $\T=\{t_1,\dots,t_p\}$ and $\V=\{1,\dots,q\}$.
\begin{itemize}
  \item An
      \emph{array-representation}
      $\cal A$ of $G$ is given by assigning
\begin{itemize}
  \item
a nonnegative integer number
$d_v$
to
each $v\in \V$, and
  \item
an array $\u A_{\alpha }$ of size
$d_{v_1}\times\dots\times
d_{v_k}$\footnote{For each sequence of nonnegative integers
      $d_1,\dots,d_q$ with
      $\min\{d_1,\dots,d_q\}=0$,
      there is exactly one array of
      size $d_1\times\dots\times d_q$. In particular, the ``empty''
      matrices of sizes $0\times n$
      and $m\times 0$ give
      the linear mappings $\F^n\to
      0$ and $0\to \F^m$.}
to each $t_{\alpha }\in \T$
      with arrows $\lambda_1,\dots,\lambda _k$ of the form
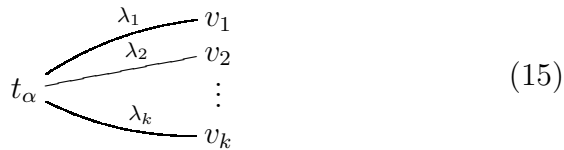
\begin{equation}\label{liw1}
\begin{split}
\xymatrix@R=-3pt@C=2cm
{
&v_1\\
&\\
&v_2\\
t_{\alpha }\ar@/^/@{-}[ruuu]^(0.6){\lambda_1}
\ar@{-}[ru]^(0.6){\lambda_2}
\ar@{}[r]
\ar@/_/@{-}[rdd]^(0.6)
{\lambda_k}
&\text{\raisebox{-3.5pt}{$\vdots$}}\\
&\\
&v_k\\
}
\end{split}
\end{equation}
(each line $\lin$ is
$\longrightarrow$ or
$\longleftarrow$ and some of
$v_1,\dots,v_k$ may coincide).
\end{itemize}
The vector $d=(d_1,\dots,d_q)$ is the \emph{dimension} of $\cal A$.
(For
      example, an
      array-representation $\cal A$ of
      \eqref{hle1} of dimension $d=(d_1,d_2)$
      is given by three arrays $\u A_1$, $\u A_2$, and $\u A_3$ of
      sizes $d_1$, $d_1\times
      d_1\times d_1$, and $d_1\times
      d_2\times d_2$.)

  \item We say that two
      array-representations ${\cal
      A}=(\u A_1,\dots,\u A_p)$ and
      ${\cal B}=(\u B_1,\dots,\u
      B_p)$ of $G$ of the same dimension $d=(d_1,\dots,d_q)$ are
      \emph{isomorphic} and write
${\cal S}:{\cal A}\too {\cal B}$
(or ${\cal A}\simeq {\cal B}$ for
short)
      if there
      exists a sequence
      ${\cal S}:=(S_1,\dots,S_q)$ of
      nonsingular matrices of sizes
      $d_1\times
      d_1,\dots,d_q\times  d_q$
      such that
\begin{equation}\label{mbv1}
(S_{v_1}^{\tau_1},\dots, S_{v_k}^{\tau_k}):\u A_{\alpha }\too \u B_{\alpha }
\end{equation}
(see \eqref{mkt}) for each $t_{\alpha }\in \T$ with arrows \eqref{liw1}, where      \[
\tau_i:=\left\{
                \begin{array}{ll}
                  1 & \hbox{if $\lambda_i: t_{\alpha}\longleftarrow  v_i$} \\
                  -T & \hbox{if $\lambda_i: t_{\alpha}\longrightarrow  v_i$}
                \end{array}
              \right.\quad
\text{for all $i=1,\dots,k,$}               \]
and $S^{-T}:=(S^{-1})^T$ (which is called the \emph{contragredient matrix} of $S$).
\end{itemize}
\end{definition}

\begin{example}\label{mkw}
\begin{itemize}
  \item Each array-representation of dimension $d=(d_1)$
      of
\begin{equation}\label{fdq}
\xymatrix{t\ar@/^/@{<-}[r]\ar@/_/@{<-}[r]&1}
\qquad\text{or}\qquad
\xymatrix{t\ar@/^/@{->}[r]\ar@/_/@{<-}[r]&1}
\end{equation}
is a $d_1\times d_1$ matrix $A=[a_{ij}]$, which
is isomorphic to an
array-representation $B=[b_{ij}]$
if and only if there exists a $d_1\times d_1$
nonsingular matrix $S=[s_{ij}]$
such that
\[
b_{i'j'}=\sum_{i,j} a_{ij}s_{ii'}s_{jj'}\qquad\text{or}\qquad
b_{i'j'}=\sum_{i,j} a_{ij}r_{ii'}s_{jj'},
\]
respectively, where $[r_{ii'}]:=S^{-T}$. Thus,
\[
B=S^TAS\qquad\text{or}\qquad B=S^{-1}AS,
\]
and so we can consider each
array-representation of
\eqref{fdq} as the matrix of a
\emph{bilinear form} or \emph{linear operator}, respectively.

  \item Each array-representation
      of dimension $d=(d_1)$
      of
\begin{equation}\label{fhdq}
\xymatrix@C=50pt{
{t}\ar@/_/@<-0.3ex>[r]
  &{1}\ar@/_/@<-0.3ex>[l]
  \ar[l]
  }
\end{equation}
is
a $d_1\times d_1\times d_1$ array $\u A=[a_{ijk}]$. It is
reduced by transformations
\[
(S,S,S^{-T}):\u A\mapsto \Big[\sum_{i,j,k} a_{ijk}s_{ii'}s_{jj'}
r_{kk'}\Big]_{i'j'k'},\qquad [r_{kk'}]:=S^{-T},
\]
in which $S=[s_{ij}]$ is a $d_1\times d_1$ nonsingular matrix.

  \item By \eqref{mbv1}, each array-representation of \eqref{liw1} defines an $(m,n)$-tensor (i.e., an $m$ times contravariant and $n$ times
covariant
tensor), where $m$ is the number of arrows $\longrightarrow$ and $n$ is the number of arrows $\longleftarrow$; see Section \ref{hyu}.
In particular, each array-representation of \eqref{fhdq} defines a $(1,2)$-tensor $T\in V^*\otimes
V^*\otimes
V$, with defines a multiplication
 in $V$ converting $V$ into a \emph{finite dimensional algebra}; see Example \ref{err}.
\end{itemize}
\noindent In Section \ref{sss3} we show that the problem \eqref{zzzz} contains the problems of classifying $(1,2)$-tensors and $(0,3)$-tensors.
\end{example}

Our main result is the following theorem (which ensures the statement \eqref{kjt3}); the other theorems are its special cases.

\begin{theorem}\label{ktw1}
Let $G$ be a
      directed bipartite graph with
      the set
      $\T=\{t_1,\dots,t_p\}$ of
      left vertices and the set
      $\V=\{1,\dots,q\}$ of right
      vertices, in which each left
      vertex has at most three
      arrows. Let
      $d=(d_1,\dots,d_q)$ be an
      arbitrary sequence of
      nonnegative integers, and let
      ${\cal X}=(\u X_1,\dots,\u
      X_p)$ be a variable
      array-representation  of $G$ of dimension $d$, in which the
      entries of arrays $\u X_1,\dots,\u
      X_p$ are independent variables.

      Then there
      exists a
      partitioned three-dimensional variable
      array $\u F({\cal X})$ in
      which
\begin{itemize}
  \item $p$ spatial blocks are $\u
      X_1,\dots,\u X_p$, and
  \item each entry of the other  spatial blocks is $0$ or $1$,
\end{itemize}
such that  two array-representations $\cal
A$ and $\cal B$  of dimension
$d$  of the graph $G$ over a field are isomorphic if and only if
\begin{equation}\label{kkt1}
\text{$\u F({\cal A})$ and $\u F({\cal B})$ are equivalent as unpartitioned arrays.}
\end{equation}
\end{theorem}

\begin{remark}\label{rema}
The variable array $\u F({\cal X})$ defines an embedding of the vector space of array-representations of dimension
$d$  of the graph $G$ into the vector space of three-dimensional arrays of some fixed size. This embedding satisfies Definition \ref{cbd} and its image (which consists of all $\u F({\cal A})$ with $\mc A$ of dimension $d$)  is an affine subspace. \emph{Two representations are isomorphic if and only if their images are equivalent.}
\end{remark}

\section{Proof of the statement \eqref{kjt1}} \label{ss2}

The statement \eqref{kjt1} is proved in the following theorem.

\begin{theorem}\label{hep}
For each partition \eqref{1jkk},
there exists a partitioned three-dimensional variable array $\u F({\u X})$ in which
\begin{itemize}
  \item[\rm(i)] one spatial block is an $m\times n\times t$ variable array
${\u X}$ whose entries  are independent variables, and
  \item[\rm(ii)] each entry of the other  spatial blocks is $0$ or $1$,
\end{itemize}
such that  two $m\times n\times t$ arrays $\u
A$ and $\u B$
partitioned into $\o m\cdot\o n\cdot\o t$ spatial blocks are block-equivalent if and only if $\u F({\u A})$ and $\u F({\u B})$ are equivalent.
\end{theorem}

\subsection{Slices and strata of three-dimensional arrays}\label{lku}

We give
a three-dimensional array $\underline
A=[a_{ijk}]_{i=1}^m{}_{j=1}^n{}_{k=1}^{t}$
by the sequence of matrices
\begin{equation}\label{kir}
\v A:=(A_{1},A_{2},\dots,A_{t}),\qquad A_{k}:=[a_{ijk}]_{ij},
\end{equation}
which are the \emph{frontal slices} of
$\underline A$. For example, a $3\times
3\times 3$ array $\underline
A=[a_{ijk}]_{i=1}
^3{}_{j=1}^3{}_{k=1}^3$ can be given by its
frontal slices
\begin{equation}\label{oir}
\begin{split}
\xymatrix@=-2pt{%
&&&&&&&&\smash{A_3\ \ }&*{}\ar@{-}[rrrrrr]\ar@{-}[dddd]&&&&&&*{}
\ar@{-}[dddddd]
  \\
&&&&&&&&&&a_{113}&&a_{123}&&a_{133}\\\\
&&&&&&&&&&a_{213}&&a_{223}&&a_{233}\\
&&&&\smash{A_2\ \ }&*{}\ar@{-}[rrrrrr]
\ar@{-}[dddd]&&&&&&*{}
\\
&&&&&&a_{112}&&a_{122}&&a_{132}&&a_{323}&&a_{333}&{}\\
&&&&&&&&&&&&&&&*{}\ar@{-}[llll]\\
&&&&&&a_{212}&&a_{222}&&a_{232}
            {}\save[]+<5cm,0cm>*{}
\restore
\\
\smash{A_1}&*{}\ar@{-}[rrrrrr]\ar@{-}[dddddd]
*{}\ar@{.}[uuuuuuuurrrrrrrr]&&&&{}&&*{}
\ar@{.}[uuuuuuuurrrrrrrr]\ar@{-}[dddddd]\\
&&a_{111} && a_{121}&& a_{131}&
&a_{322}&&a_{332}
\\&&&&&&&{}\ar@{-}[rrrr]&&&&*{}\ar@{-}[uuuuuu]\\
&&a_{211} && a_{221}&& a_{231}& \\\\
&&a_{311} && a_{321}&& a_{331}& \\
&*{}\ar@{-}[rrrrrr]&&&&&&*{}\ar@{.}[uuuuuuuurrrrrrrr]\\
}
\end{split}
\end{equation}
An array $\underline
A=[a_{ijk}]_{i=1}^m{}_{j=1}^n{}_{k=1}^{t}$ can be also given by the sequence of
\emph{lateral slices} $[a_{i1k}]_{ik},\dots,
[a_{ink}]_{ik}$, and by
the sequence of \emph{horizontal slices}
$[a_{1jk}]_{jk},\dots,[a_{mjk}]_{jk}$.

A \emph{linear reconstruction} of a sequence $(A_1,\dots,A_t)$
of matrices of the
same size given by a nonsingular matrix
$U=[u_{ij}]$ is the transformation
\begin{equation}\label{dgj}
(A_1,\dots,A_t)\ci U:=(A_1u_{11}+\dots+A_tu_{t1},\dots,
A_1u_{1t}+\dots+A_tu_{tt}).
\end{equation}
Clearly, every linear reconstruction of $(A_1,\dots,A_t)$
is a sequence of the following
elementary linear reconstructions:
\begin{itemize}
  \item[(a)]  interchange of two
      matrices,

  \item[(b)] multiply any
      matrix by a nonzero element
      of $\F$,

  \item[(c)] add a matrix multiplied by an element
      of $\F$ to another matrix.
\end{itemize}

The following lemma is obvious.

\begin{lemma}\label{lkg}
Given two three-dimensional arrays $\underline A$ and $\underline B$ of the same size.
\begin{itemize}
  \item[\rm(a)] $\underline A$ and
      $\underline B$ are equivalent
      if and only if $\underline B$
      can be obtained from
      $\underline A$ by
      linear reconstructions of frontal slices, then of lateral slices, and finally of horizontal
      slices.
  \item[\rm(b)] $(R,S,U): \u A \too
      \u B$ if and only if
\begin{equation}\label{krw}
\v B=(R^TA_{1}S,R^TA_{2}S,\dots,R^TA_{t}S)\ci U.
\end{equation}
\end{itemize}
\end{lemma}

Let $\underline A =[a_{ijk }]_{i=1}^{m}\,{}_{j=1}^{n}\,{}_{k=1}^{ t}$ be a three-dimensional array, whose partition
$\u A=[\u
A_{\alpha \beta \gamma }]_{\alpha
=1}^{\o m}\,{}_{\beta=1}^{\o
n}\,{}_{\gamma=1}^{\o t}
$ into $\o m\cdot \o
n\cdot \o t$ spatial blocks is defined by partitions \eqref{1jkk} of its index sets.
The partition  of $\{1,\dots,t\}$ in \eqref{1jkk} into $\o t$ disjoint subsets defines also the division of $\u A$ by frontal planes into $\bar t$ \emph{frontal strata}
\[
[a_{ijk }]_{i=1}^{m}\,{}_{j=1}^{n}\,
{}_{k=1}^{k_1},\ \
[a_{ijk }]_{i=1}^{m}\,{}_{j=1}^{n}\,
{}_{k=k_1+1}^{k_2},\ \ \dots,\ \
[a_{ijk }]_{i=1}^{m}\,{}_{j=1}^{n}\,
{}_{k=k_{\o t -1}+1}^{k_{\o t}};
\]
each frontal stratum is the union of frontal slices corresponding to the same subset. In the same way, $\underline A$ is divided into $\o n$ \emph{lateral strata} and into $\o m$ \emph{horizontal strata}.

Two partitioned three-dimensional arrays $\underline A$
and $\underline B$ are
block-equivalent if and only if $\underline B$ can be
obtained from $\underline A$ by
linear reconstructions of
frontal strata, of lateral strata, and of
horizontal strata.

\subsection{The lemma that implies Theorem \ref{hep}}\label{gfsg}

\begin{lemma}\label{jhp}
For each partition \eqref{1jkk},
there exists a three-dimensional variable array $\u G({\u X})$ partitioned into $\o m'\cdot\o n'\cdot\o t'$ spatial blocks and satisfying  {\rm(i)} and  {\rm(ii)} from Theorem \ref{hep}
such that two $m\times n\times t$ arrays $\u
A$ and $\u B$
partitioned into $\o m\cdot\o n\cdot\o t$ spatial blocks are block-equivalent if and only if $\u G({\u A})$ and $\u G({\u B})$ are block-equivalent with respect to partition into $\o m'\cdot\o n'\cdot t'$ spatial blocks (i.e., we delete the horizontal partition).
\end{lemma}

This lemma implies Theorem \ref{hep} since we can delete the horizontal partition, then the lateral partition, and finally the frontal partition in the same way.

\subsection{Proof of Lemma \ref{jhp} for a partitioned array of size
$m\times n\times 6$}\label{gfs}

In order
to make the proof of Lemma \ref{jhp} clearer, we first
prove it for arrays of size
$m\times n\times 6$ partitioned by a
frontal plane into two spatial blocks of
size $m\times n\times 3$.

Such an array $\u A$ can be given by the
sequence
\begin{equation*}\label{fel}
\v A=(A_1,A_2,A_3\,|\,A_4,A_5,A_6)
\end{equation*}
of its frontal slices, which are $m\times n$ matrices (see \eqref{kir}).  Its two spatial blocks are given by the sequences $(A_1,A_2,A_3)$ and $(A_4,A_5,A_6)$. Construct by
$\u A$ the unpartitioned array $\u M^A$
given by the sequence of frontal slices
\begin{equation}\label{jue}
\begin{split}
&\v M^A=(M^A_1,\dots,M^A_6)\\
&:=\left(
\left[
\begin{MAT}(@){ccccccc}
I_r&0&0&&&&0\\
&&&0&0&0&\\
0&&&&&&A_1
\addpath{(0,2,2)rrrrrrdd}
\addpath{(3,3,2)ddrrrr}
\\
\end{MAT}
\right],
           \,
\left[
\begin{MAT}(@){ccccccc}
0&I_r&0&&&&0\\
&&&0&0&0&\\
0&&&&&&A_2
\addpath{(0,2,2)rrrrrrdd}
\addpath{(3,3,2)ddrrrr}
\\
\end{MAT}
\right],
           \,
\left[
\begin{MAT}(@){ccccccc}
0&0&I_r&&&&0\\
&&&0&0&0&\\
0&&&&&&A_3
\addpath{(0,2,2)rrrrrrdd}
\addpath{(3,3,2)ddrrrr}
\\
\end{MAT}
\right],\right.
                                \\
&\phantom{:=} \left.\left[
\begin{MAT}(@){ccccccc}
0&0&0&&&&0\\
&&&I_{2r}&0&0&\\
0&&&&&&A_4
\addpath{(0,2,2)rrrrrrdd}
\addpath{(3,3,2)ddrrrr}
\\
\end{MAT}
\right],
           \,
\left[
\begin{MAT}(@){ccccccc}
0&0&0&&&&0\\
&&&0&I_{2r}&0&\\
0&&&&&&A_5
\addpath{(0,2,2)rrrrrrdd}
\addpath{(3,3,2)ddrrrr}
\\
\end{MAT}
\right],
           \,
\left[
\begin{MAT}(@){ccccccc}
0&0&0&&&&0\\
&&&0&0&I_{2r}&\\
0&&&&&&A_6
\addpath{(0,2,2)rrrrrrdd}
\addpath{(3,3,2)ddrrrr}
\\
\end{MAT}
\right]
\right),
\end{split}\end{equation}
in which
\begin{equation}\label{luw}
 r:=\min\{m,n\}+1.
\end{equation}
Let $\v
B:=(B_1,B_2,B_3\,|\,B_4,B_5,B_6)$ give another array $\u B$ of the same size as $\u A$, which is partitioned by the frontal plane
conformally to $\u A$. Let us prove
that
\begin{equation}\label{kjt}
\begin{split}
  & \text{$\u A$ and $\u B$ are block-equivalent} \\
    \Longleftrightarrow\ \ & \text{$\u M^A$ and $\u M^B$ are equivalent.}
\end{split}
\end{equation}

$\Longrightarrow$.
In view of Lemma \ref{lkg}(b), $\u B$ can be obtained from $\u A$ by a sequence of the following transformations:
\begin{itemize}
  \item[(i)]
simultaneous equivalence
transformations of $A_1,\dots,A_6$,
  \item[(ii)]  linear reconstructions (a)--(c) (from Section \ref{lku}) of $(A_1,A_2,A_3)$,
  \item[(iii)] linear reconstructions (a)--(c) of $(A_4,A_5,A_6)$.
\end{itemize}
It suffices to consider the case when
$\u B$ is obtained from $\u A$ by one of the transformations (i)--(iii).

If this transformation is (i), then $\u M^B$ can be obtained from $\u M^A$ by a simultaneous equivalence transformation of its frontal slices $M^A_1,\dots,M^A_6$. Hence, $\u M^A$ and $\u M^B$ are equivalent.

If this transformation is (ii), then we make a linear reconstruction of $(M^A_1,M^A_2,M^A_3)$. It spoils the blocks $ [I_r\,0\,0],$ $
[0\,I_r\,0],$ $ [0\,0\,I_r]$ in \eqref{jue}, which
can be restored by simultaneous
elementary transformations of
$M_1^A,\dots,M_6^A$ that do not change the new $A_1,\dots,A_6$. Hence $\u M^A$ and $\u M^B$ are equivalent.

The case of transformation (iii) is considered analogously.
\medskip

$\Longleftarrow$. Let $\u M^A$ and
$\u M^B$ be equivalent. By Lemma 4(b), there exists a matrix sequence
$\v N=(N_1,\dots,N_6)$ such that the matrices in $\v M^A$ and $\v N$
are simultaneously equivalent and $\v N$ is reduced to $\v
M^B$ by some linear reconstruction $\v N\ci S^{-1}=\v M^B$ given by
a nonsingular $6\times 6$ matrix $S=[s_{jk}]$ (see
\eqref{dgj}). Then
\[
(N_1,\dots,N_6)=(M^B_1,\dots,M^B_6)\ci S
\]
and so
\begin{equation}\label{deo}
N_k=\left[
\begin{MAT}(@){ccccccc}
I_rs_{1k}&I_rs_{2k}&I_rs_{3k}&&&&0\\
&&&I_{2r}s_{4k}&I_{2r}s_{5k}&I_{2r}s_{6k}&\\
0&&&&&&C_k
\addpath{(0,2,2)rrrrrrdd}
\addpath{(3,3,2)ddrrrr}
\\
\end{MAT}
\right],\quad k=1,\dots,6,
\end{equation}
where
\begin{equation}\label{lue}
(C_1,\dots,C_6):=(B_1,\dots,B_6)\ci S.
\end{equation}

Since the matrices in $\v M^A$ are simultaneously
equivalent to the matrices in $\v N$, $\rank
M_k^A=\rank N_k$ for all $k$. Since $A_k,\,B_k,\,C_k$ are
$m\times n$, \eqref{luw} shows that
$\rank A_k<r$ and $\rank C_k<r$.
If $k=1,2,3$, then $\rank N_k=\rank M^A_k<2r$, and so by \eqref{deo} $s_{4k}= s_{5k}=s_{6k}=0$. If $k=4,5,6$, then
$\rank N_k=\rank M^A_k\ge 2r$ and so $(s_{4k},s_{5k},s_{6k})\ne (0,0,0)$; furthermore, $\rank N_k=\rank M^A_k< 3r$ and so $s_{1k}= s_{2k}=s_{3k}=0$.
Thus,
\begin{equation}\label{fwl}
S=\begin{bmatrix}
    s_{11} & s_{12}&s_{13} \\
    s_{21} & s_{22}&s_{23} \\
  s_{31} & s_{32}&s_{33} \\
  \end{bmatrix}\oplus
\begin{bmatrix}
    s_{44} & s_{45}&s_{46} \\
    s_{54} & s_{55}&s_{56} \\
  s_{64} & s_{65}&s_{66} \\
  \end{bmatrix}.
\end{equation}

By \eqref{lue} and \eqref{fwl}, the array $\u B$ is block-equivalent to the conformally
partitioned array $\u C$
given by $\v
C=(C_1,C_2,C_3|C_4,C_5,C_6)$. It remains to prove that the arrays $\u A$ and $\u C$ are
block-equivalent.
By \eqref{deo},
\[
N_k=M_k^CQ,\quad k=1,\dots,6,
\]
where
\[
Q:=\begin{bmatrix}
    s_{11}I_{r} & s_{21}I_{r}&s_{31}I_{r} \\
    s_{12}I_{r} & s_{22}I_{r}&s_{32}I_{r} \\
  s_{13}I_{r} & s_{23}I_{r}&s_{33}I_{r} \\
  \end{bmatrix}\oplus
  \begin{bmatrix}
    s_{44}I_{2r} & s_{54}I_{2r}&s_{64}I_{2r} \\
    s_{45}I_{2r} & s_{55}I_{2r}&s_{65}I_{2r} \\
  s_{46}I_{2r} & s_{56}I_{2r}&s_{66}I_{2r} \\
  \end{bmatrix}\oplus I_n.
\]
Therefore, the matrices in
$\v N$ and $\v M^C$ are
simultaneously equivalent. Since the matrices in  $\v
M^A$ and $\v N$ are simultaneously
equivalent, we have that the matrices in
\begin{equation}\label{eku}
\v M^A=(M^A_1,\dots,M^A_6)\quad\text{and}\quad \v M^C=(M^C_1,\dots,M^C_6)
\end{equation}
are simultaneously equivalent. Hence,
the representations
\[
{\xymatrix@C=45pt{
 {\F^{9r+n}}
\ar@/^2pc/[r]^{M^A_1}
\ar@/^/[r]^{M^A_2}_{\vdots}
\ar@/_1.5pc/[r]_{M^A_6}
   &{\F^{3r+m}}
  }}\qquad \text{and}\qquad
{\xymatrix@C=45pt{
 {\F^{9r+n}}
\ar@/^2pc/[r]^{M^C_1}
\ar@/^/[r]^{M^C_2}_{\vdots}
\ar@/_1.5pc/[r]_{M^C_6}
   &{\F^{3r+m}}
  }}
\]
of the quiver
\begin{equation}\label{drh}
\xymatrix@C=45pt@1{
 {\cim}
\ar@/^0.9pc/[r]
\ar@/^0.4pc/[r]
\ar@/_0.9pc/[r]^{\vdots}
   &{\cim}}\qquad (6\text{ arrows})
\end{equation}
are isomorphic.

By \eqref{jue},
the sequences \eqref{eku} have the form
\[
(E_1,\dots,E_6)\oplus (A_1,\dots,A_6)
\quad\text{and}\quad
(E_1,\dots,E_6)\oplus (C_1,\dots,C_6).
\]
By the \emph{Krull--Schmidt theorem} (see
\cite[Corollary 2.4.2]{haz-kir} or
\cite[Section 43.1]{hor-ser}), each
representation of a quiver is
isomorphic to a direct sum of
indecomposable representations; this sum
is uniquely determined, up to
permutation and isomorphisms of direct
summands. Hence, the sequences
$(A_1,\dots,A_6)$ and $(C_1,\dots,C_6)$
gave isomorphic representations of the
quiver \eqref{drh}, and so their matrices are simultaneously
equivalent. Thus, $\u A$ and $\u C$ are block-equivalent, which proves that
$\underline A$ and $\underline B$ are
block-equivalent.
\medskip

\subsection{Proof of Lemma \ref{jhp} for an arbitrary partitioned three-dimensional array}\label{gfs3}

Let us
prove that Lemma \ref{jhp} holds for a
partitioned array $\u A=[\u A_{\alpha \beta \gamma
}]_{\alpha =1}^{\o m}\,{}_{\beta=1}^{\o
n}\,{}_{\gamma=1}^{\o t} $ of
size $m\times n\times t$ whose
partition into $\o m\cdot \o n\cdot \o
t$ spatial blocks is given by \eqref{1jkk}.
There is nothing to prove if $\o t=1$.
Assume that $\o t\ge 2.$

We give $\u A$ by the sequence
\begin{equation*}\label{jde}
\v A=(A_1,\dots,A_{k_1}|A_{k_1+1},\dots,A_{k_2}|
\dots|A_{k_{{\o t}-1}+1},\dots,A_{k_{\o t}}),\quad
k_{\o t}=t
\end{equation*}
of frontal slices
\[A_1=[A_{\alpha \beta 1}]_{\alpha =1}^{\o m}\,{}_{\beta=1}^{\o
n},\
A_2=[A_{\alpha \beta 2}]_{\alpha =1}^{\o m}\,{}_{\beta=1}^{\o
n},\ \dots\]
They are
block matrices of size $m\times n$ with the same partition into $\o m\cdot\o n$
blocks. By analogy with \eqref{jue}, we consider
the array $\u M^A$
given by the sequence $\v
M^A=(M_1^A,\dots,M_{t}^A)$ of block matrices
\[
M_k^A:=\begin{bmatrix}
\Delta_{k1}\ \dots\ \Delta_{kk_1}& &&&0 \\
  & \Delta_{k,{k_1}+1}\ \dots\ \Delta_{k{k_2}}&&& \\
  &&\ddots&&\\\
  &&&\Delta_{k,{k_{\o t-1}+1}}\ \dots\ \Delta_{kk_{\o t}}&\\
  0\ \ \phantom{\ \dots\ \Delta_{kk_1}}&&&&A_k
\end{bmatrix}
 \]
in which the blocks
\[
\underbrace{\Delta_{k1},\ \dots,\ \Delta_{kk_1}}_{\text{\normalsize each of size
$r\times r$}}\hspace{-4pt},
         \quad
\underbrace{\Delta_{k,{k_1}+1},\ \dots,\ \Delta_{k{k_2}}}_{\text{\normalsize each of size $2r\times 2r$}},
     \
\dots,
      \
\underbrace{\Delta_{k,{k_{\o t-1}+1}},\ \dots,\ \Delta_{kk_{\o t}}}_{{\text{\normalsize each of size
$2^{\o t-1}r\times 2^{\o t-1}r$}}}
\]
are defined as follows:    $r:=\min\{m,n\}+1$ and
\begin{equation*}\label{juy}
\begin{array}{c}
(\Delta_{11},\dots,\Delta_{1t}):=(I,0,\dots,0)\\
(\Delta_{21},\dots,\Delta_{2t}):=(0,I,\dots,0)\\
\hdotsfor{1}\\
\,(\Delta_{t1},\dots,\Delta_{tt}):=(0,\dots,0,I).
\end{array}
\end{equation*}
We partition $\u M^A$ by lateral and
horizontal planes that extend the
partition of its spatial block $\u A$, but
we do not partition $\u M^A$ by frontal
planes. Thus, each matrix $M_k^A$ is
partitioned as follows:
\begin{equation}\label{vei}
M_k^A=
\left[\begin{array}{cccccc|c|c|c}
      \Delta_{k1}&\Delta_{k2}&\dots&0&0&0&0&\vdots&0 \\
      \dots&\dots&\dots&\dots&\dots&\vdots&\vdots&\vdots&\vdots \\
      0&0&\dots&\Delta_{k,t-1}&\Delta_{kt}&0&0&\vdots&0 \\
 0&0&\dots&0&0&A_{11k}&A_{12k}&\dots&A_{1{\o n}k} \\\hline
0&0&\dots&0&0&A_{21k}&A_{22k}&\dots&A_{2{\o n}k} \\\hline
\vdots&\vdots&&\vdots&\vdots&\vdots&\vdots&&\vdots
\\\hline
0&0&\dots&0&0&A_{{\o m}1k}&A_{{\o m}2k}&\dots&A_{{\o m}{\o n}k}
    \end{array}\right].
\end{equation}

Let $\u B=[\u B_{\alpha \beta \gamma
}]_{\alpha =1}^{\o m}\,{}_{\beta=1}^{\o
n}\,{}_{\gamma=1}^{\o t}$ be an array
of the same size and with the same partition into spatial blocks as $\u A$. Let us prove that
\begin{equation}\label{klj}
\begin{split}
  & \text{$\u A$ and $\u B$ are block-equivalent} \\
    \Longleftrightarrow\ \ & \text{$\u M^A$ and $\u M^B$ are block-equivalent.}
\end{split}
\end{equation}

$\Longrightarrow$. It is proved as for
\eqref{kjt}.

$\Longleftarrow$.  Let $\u M^A$ and
$\u M^B$ be block-equivalent. Then there exists a sequence $\v N=(N_1,\dots,N_t)$ of matrices of the same size and with the same partition as \eqref{vei} such that the matrices in  $\v
M^A$ and $\v N$ are simultaneously block-equivalent
and $\v N\ci S^{-1}=\v M^B$ for some
nonsingular matrix $S=[s_{jk}]\in\F^{t\times t}$. Thus,
\[
(N_1,\dots,N_t)=(M^B_1,\dots,M^B_t)\ci S
\]
and so
\begin{multline*}
N_k=
\begin{bmatrix}I_rs_{1k}
&\dots&I_rs_{k_1k}\end{bmatrix}\oplus
\begin{bmatrix}I_{2r}s_{k_1+1,k}
&\dots&I_{2r}s_{k_2k}\end{bmatrix}\\
\oplus\dots\oplus
\begin{bmatrix}I_{2^{{\o t}-1}r}s_{k_{{\o t}-1}+1,k}&
\dots&I_{2^{{\o t}-1}r}s_{k_{\o t}k}\end{bmatrix}
\oplus C_k,
\end{multline*}
where
\begin{equation}\label{lue1}
\v C=(C_1,\dots,C_t):=(B_1,\dots,B_t)\ci S.
\end{equation}
Denote by $\u C$
the array defined by \eqref{lue1} and partitioned
conformally with the partitions of $\u
A$ and $\u B$.

Reasoning as in Section \ref{gfs}, we prove that
\[
S=S_1\oplus S_2\oplus\dots\oplus S_{\o t},\qquad S_1\in\F^{k_1\times k_1},\
S_2\in\F^{(k_2-k_1)\times
(k_2-k_1)},\ \dots.
\]
Thus, the arrays
$\u C$ and $\u B$ are block-equivalent.
It remains to prove that
\begin{equation}\label{dtw}
\text{the arrays $\u A$ and $\u C$ are block-equivalent.}
\end{equation}

We can reduce $\v N$ to $\v
M^C=(M^C_1,\dots,M^C_{t})$
by simultaneous elementary
transformations of columns
$1,\dots,k_1$ of $N_1,\dots,N_t$, simultaneous elementary
transformations of columns
$k_1+1,\dots,k_2$, \dots, and simultaneous elementary
transformations of columns
$k_{\o t-1}+1,\dots,k_{\o
t}$ (these transformations do not change $\u C$). Hence, the matrices in $\v
N$ and $\v M^C$ are simultaneously block-equivalent. Since the matrices in  $\v M^A$ and $\v N$
are simultaneously block-equivalent, we have
that
\begin{equation}\label{dfr}
\parbox[c]{0.8\textwidth}{the matrices in $\v M^A$ and $\v M^C$ are simultaneously block-equivalent.
}
\end{equation}

By the \emph{block-direct sum} of two
block matrices $M=[M_{\alpha \beta
}]_{\alpha =1}^p{}_{\beta =1}^q$ and
$M=[N_{\alpha \beta }]_{\alpha
=1}^p{}_{\beta =1}^q$, we mean the block
matrix
\[
M\boxplus N:=[M_{\alpha \beta }\oplus N_{\alpha \beta }]_{\alpha =1}^p{}_{\beta =1}^q.
\]
This operation was studied in
\cite{ser}; it is naturally extended to
$t$-tuples of block matrices:
\[
(M_1,\dots,M_t)\boxplus (N_1,\dots,N_t):=(M_1\boxplus N_1,\dots, M_t\boxplus N_t).
\]

By \eqref{vei},
$M^A_k=\Delta_k\boxplus A_k$ for $k=1,\dots,t$, where
\[
\Delta _k:=
\begin{MAT}(@){4ccccc2c2c2c2c4}
  \first4
      \Delta_{k1}& \Delta_{k1}&\dots&\dots&0&&&& \\
    \vphantom{A}\dots&\dots&\dots&\dots&\dots&&&& \\
      0&\dots&\dots&\Delta_{k,t-1}&\Delta_{kt}&&&&
\\2
&&&&&&&& \\2
&&&&&&&& \\2
&&&&&&&& \\2
&&&&&&&& \\4
\end{MAT}
\]
(the empty strips do not have entries). Thus, $\v
M^A=\v \Delta\boxplus \v A$.

Each matrix $M^C_k$ has the form
\eqref{vei} with $C_{ijk}$ instead of
$A_{ijk}$, and so $\v M^C=\v
\Delta\boxplus \v C$.

Let us prove that the matrices in $\v A$ and
$\v C$ are simultaneously block-equivalent.
Define the quiver $Q$ with $\o m+\o n$
vertices $1,\dots,\o m,1',\dots,\o n'$
and with $\o m\o nt$ arrows: with $t$ arrows
\begin{equation}\label{kkg}
{\xymatrix@C=100pt{
 {\beta'}
\ar@/^2pc/[r]^{\lambda_{\alpha \beta 1}}
\ar@/^/[r]^{\lambda_{\alpha \beta 2}}_{\vdots}
\ar@/_1.5pc/[r]_{\lambda_{\alpha \beta t}}
   &{\alpha}
  }}
\end{equation}
from each vertex $\beta'
\in\{1',\dots,\o n'\}$ to each vertex
$\alpha\in\{1,\dots,\o m\}$.

Let $\v K=(K_1,\dots,K_{t})$ be an arbitrary sequence of block matrices
$K_k=[K_{\alpha \beta k}]_{\alpha
=1}^{\o m}{}_{\beta =1}^{\o n}$, in which $K_{\alpha \beta k}$ is of
size $m_{\alpha}\times n_{\beta}$. This sequence defines the array $\u K$ partitioned into $\o m\cdot\o n\cdot 1$ spatial blocks. Define the representation
$R(\v K)$ of $Q$ by assigning mappings to the arrows
\eqref{kkg} as follows:
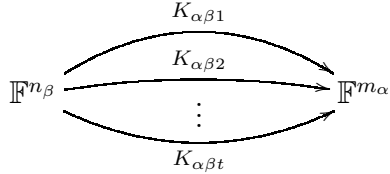
\[
{\xymatrix@C=100pt{
 {\F^{n_{\beta }}}
\ar@/^2pc/[r]^{K_{\alpha \beta 1}}
\ar@/^/[r]^{K_{\alpha \beta 2}}_{\vdots}
\ar@/_1.5pc/[r]_{K_{\alpha \beta  t}}
   &{\F^{m_{\alpha }}}
  }}
\]

Let $\v L=(L_1,\dots,L_{t})$ be
another sequence of block matrices $L_k=[L_{\alpha \beta k}]_{\alpha
=1}^{\o m}{}_{\beta =1}^{\o n}$, in which all matrices have the same size and the same partition into
blocks as the matrices of $\v K$. Clearly, the matrices in  $\v K$ and $\v L$ are
simultaneously block-equivalent if and
only if the representations $R(\v K)$
and $R(\v L)$ are isomorphic.

By \eqref{dfr}, the matrices in  $\v M^A=\v
\Delta\boxplus \v A$  and $\v M^C=\v
\Delta\boxplus \v C$ are simultaneously
block-equivalent. Hence, the representations
$R(\v M^A)=R(\v \Delta)\oplus R(\v A)$
and $R(\v M^C)=R(\v \Delta)\oplus R(\v
C)$ of $Q$ are isomorphic. By the
Krull--Schmidt theorem for quiver
representations, $R(\v A)$  and $R(\v
C)$ are isomorphic too. Thus, the matrices in  $\v A$
and $\v C$ are simultaneously
block-equivalent, which proves \eqref{dtw}.

We have proved \eqref{klj}, which
finishes the proof of Lemma \ref{jhp} and hence the proof of  Theorem \ref{hep}.

\section{Proof of the statement (\ref{kjt2})}
\label{sss3}

In this section, we prove several
corollaries of Theorem \ref{hep}.

\begin{corollary}\label{ddr}
Theorem \ref{ktw1} holds if and only if it holds with the condition
\begin{equation}\label{kkt2}
\text{$\u F({\cal A})$ and $\u F({\cal B})$ are block-equivalent}
\end{equation}
instead of \eqref{kkt1}.
\end{corollary}

\begin{proof}
Suppose  Theorem \ref{ktw1} with \eqref{kkt2} instead of \eqref{kkt1} holds for some
      partitioned three-dimensional variable
      array $\u F({\cal X})$. Reasoning as in
      Section \ref{gfs3}, we first construct an array $\u F_1({\cal X}):=\u M^{F({\cal X})}$ that is not partitioned by frontal planes and satisfies the conditions of Theorem \ref{ktw1} with \eqref{kkt2} instead of \eqref{kkt1}. Then we apply an analogous construction to $\u F_1({\cal X})$ and obtain an array $\u F_2({\cal X})$  that is not partitioned by frontal and lateral planes and satisfies the conditions of Theorem \ref{ktw1} with \eqref{kkt2} instead of \eqref{kkt1}. At last, we
      construct an unpartitioned array $\u F_3({\cal X})$ that satisfies the conditions of Theorem \ref{ktw1}.
\end{proof}

We draw
\begin{equation}\label{jg7}
\begin{split}
\xymatrix@R=8pt@C=6pt{
&&*{}\ar@{-}'[rrr]&&&*{}\\
*{}\ar@{-}[rru]|{\,U\,}
\ar@{-}[rrr]|{S}
\ar@{-}[dd]|{R\vphantom{A^A_A}}
&&&*{}&&\\
&&&&&*{}\\
*{}\ar@{}[rrruu]|{\u {\pmb A}}
\ar@{-}'[rrr]'[rrruu]'
[rrruurru]'[rrrrru]'[rrr]
&&&*{}&&\\
}
                \quad
\xymatrix@=10pt{
\\
\too\\
}
                     \quad
\xymatrix@R=8pt@C=6pt{
&&*{}\ar@{-}'[rrr]&&&*{}\\
*{}\ar@{-}[rru]
\ar@{-}[rrr]
\ar@{-}[dd]
&&&*{}&&\\
&&&&&*{}\\
*{}\ar@{}[rrruu]|{\underline {\pmb B}}
\ar@{-}'[rrr]'[rrruu]'
[rrruurru]'[rrrrru]'[rrr]
&&&*{}&&\\
}
\end{split}
\end{equation}
if $(R,S,U): \u A \too \u B$.
We denote arrays in illustrations by  bold letters.

Gelfand and Ponomarev \cite{gel-pon}
proved that the problem of classifying
pairs of matrices up to
simultaneous similarity contains the problem of classifying $p$-tuples of matrices up to simultaneous similarity  for an arbitrary $p$. A three-dimensional analogue of their statement is the following corollary, in which we use the notion ``contains'' in the sense of Definition \ref{cbd}.

\begin{corollary}\label{dek12}
The problem \eqref{zzzz} contains
the problem of classifying $p$-tuples of
three-dimensional arrays up to
simultaneous equivalence
for an arbitrary $p$.
\end{corollary}

\begin{proof} Due to Theorem \ref{hep}, it suffices to prove that the second problem is contained in the problem of classifying partitioned three-dimensional
arrays up to block-equivalence.

Let $\mathcal A=(\u A_1,\dots,\u A_p)$
be a sequence of unpartitioned arrays
of size $m\times n\times t$.
Define the partitioned array
\[
\u N^{\cal A}=[\u N^{\cal A}_{\alpha \beta\gamma }]_{\alpha =1}^{2}{}_{\beta =1}^{2}{}_{\gamma =1}^p=
\xymatrix@R=0pt@C=10pt
{
&&&&&&*{}\ar@{.}[d]&&&*{}&&&*{}&\\
&&&&&&*{}&&&
*{}\ar@{.}'[u]'[urrr][urrrd]
&&&*{}&\\
&&&&&&&&&&&&&\\
&&&&&&&&&&&&&\\
&&&&*{}\ar@{.}'[d][drrr]
\ar@{-}[uuuurrrrr]|{\,I_t}
\ar@{-}'[rrr]'[rrrrruuuu]'[rrrrruuuulll]
[rrrrruuuullllldddd]
&&&*{}&&&*{}\ar@{.}[d]&&&\\
&&&&&&&
*{}\ar@{-}'[rrr]'[rrrrruuuu]'[rrrrruuuulll]
'[rrrrruuuullllldddd]
\ar@{.}'[u]'[urrr][urrruuuurr]
&&&*{}\ar@{-}'[dddd]'[rr][rruuuu]
&&*{}&\\
&&&&&&&&&&&&&\\
&&&&&&&&&&&&&\\
&&*{}\ar@{.}'[d][drrr]\ar@{--}[rruuuu]&&&*{}\ar@{--}'[rruuuu]'
&&&*{}&&&&&\\
&&&&&*{}\ar@{--}[rruuuu]
*{}\ar@{.}'[u]'[urrr][urrruuuurr]
&&*{}\ar@{-}[uuuu]
\ar@{}[uuuurr]|{\smash{\pmb{\u A_p}}}
&*{}\ar@{--}[rruuuu]
\ar@{-}[rr]&&*{}
&&&\\
&&&&&&&&&&&&&\\
&&&&&&&&&&&&&\\
*{}\ar@{-}[uuuurrrrr]|{\,I_t}
\ar@{-}'[rrr]'[rrrrruuuu]'[rrrrruuuulll]
[rrrrruuuullllldddd]
&&&*{}&&&*{}\ar@{.}'[rruuuu][rruuu]
&&&&&&&\\
*{}\ar@{.}'[rrrrrruuuuuuuuuuuu]
[rrrrrrrrrruuuuuuuuuuuu]
\ar@{.}[rrr]
&&&*{}
\ar@{-}'[rruuuu]'[rruuuurrr]
'[rruuuurrrddddll]'[]'[dddd]
'[ddddrrr][rrr]
\ar@{.}'[u]'[urrr]'[rrr]
&&&*{}&*{}&*{}\ar@{--}[rruuuu]&&&&&\\
&&&&&&&&&&&&&\\
&&&&&&&&&&&&&\\
&&&&&&&&&&&&&\\
*{}&&&
*{}\ar@{}[uuurrr]|{\smash{\pmb{\u A_1}}}
\ar@{.}'[lll][llluuuuu]
&&&*{}
\ar@{-}'[rruuuu][rruuuuuuuu]
&&&&&&&\\
  }
  \]
in which
\[
\u N^{\cal A}_{111}=\dots= \u N^{\cal A}_{11p}=I_t\text{ of size }1\times t\times t,\qquad
(\u N^{\cal A}_{221},\dots,
\u N^{\cal A}_{22p})=\mathcal A,
\]
and the other spatial blocks are zero (the
spatial blocks are indexed as the entries
in \eqref{oir}; the diagonal lines in
$\u N^{\cal A}_{111}$ and $\u N^{\cal
A}_{11p}$ denote the main diagonal of
$I_n$ consisting of units).

Let $\mathcal B=(\u B_1,\dots,\u B_p)$
be another sequence of unpartitioned
arrays of the same size $m\times
n\times t$. Let us prove that the arrays in $\mathcal
A$ and $\mathcal B$ are simultaneously
equivalent if and only if $\u N^{\cal
A}$ and $\u N^{\cal B}$ are
block-equivalent.

$\Longrightarrow$. It is obvious.

$\Longleftarrow$. Let $\u N^{\cal
A}$ and $\u N^{\cal B}$ be
block-equivalent; that is, there exists
\[
([r]\oplus R,S_1\oplus S_2,U_1\oplus\dots\oplus U_p): \u N^{\cal A} \too \u N^{\cal
B}.
\]
In the
notation \eqref{jg7},
\[
\xymatrix@R=0pt@C=7pt
{
&&&&&&*{}\ar@{.}[d]&&&*{}&&&*{}&\\
&&&&&&*{}&&&
*{}\ar@{.}'[u]'[urrr]'[urrrd]
&&&*{}&\\
&&&&&&&&&&&&&\\
&&&&&&&&&&&&&\\
&&&&*{}\ar@{.}'[d][drrr]
\ar@{-}[uuuurr]|{U_p}
\ar@{-}[uuuurrrrr]|{\,I_t}
\ar@{-}'[rrr]'[rrrrruuuu]'[rrrrruuuulll]
&&&*{}&&&*{}\ar@{.}[d]&&&\\
&&&&&&&
*{}\ar@{-}'[rrr]'[rrrrruuuu]'[rrrrruuuulll]
[rrrrruuuullllldddd]
\ar@{.}'[u]'[urrr]'[urrruuuurr]
&&&*{}\ar@{-}'[dddd]'[rr][rruuuu]
&&*{}&\\
&&&&&&&&&&&&&\\
&&&&&&&&&&&&&\\
&&*{}\ar@{.}'[d][drrr]
\ar@{--}'[rruuuu]'&&&*{}\ar@{--}'[rruuuu]'
&&&*{}&&&&&\\
&&&&&*{}\ar@{--}'[rruuuu]'
*{}\ar@{.}'[u]'[urrr]'[urrruuuurr]
&&*{}\ar@{-}[uuuu]
\ar@{}[uuuurr]|{\smash{\pmb{ {\u A}_p}}}
&*{}\ar@{--}'[rruuuu]'
\ar@{-}[rr]&&*{}
&&&\\
&&&&&&&&&&&&&\\
&&&&&&&&&&&&&\\
*{}
\ar@{-}[uuuurrrrr]|{\,I_t}
\ar@{-}[rrr]|{S_1}
\ar@{-}[uuuurr]|{U_1}
&&&*{}\ar@{-}'[uuuurr]
'[uuuul]
&&&*{}\ar@{.}[lll]|{S_2}
\ar@{.}'[rruuuu]'[rruuu]
&&&&&&&\\
*{}\ar@{.}[u]
\ar@{.}'[rrrrrruuuuuuuuuuuu]
'[rrrrrrrrrruuuuuuuuuuuu]
\ar@{.}[rrr]
&&&*{}
\ar@{-}'[rruuuu]'[rruuuurrr]
'[rruuuurrrddddll]'[]'[dddd]
'[ddddrrr]'[rrr]
\ar@{.}[u]
&&&*{}\ar@{.}[u]&*{}&*{}\ar@{--}'[rruuuu]'&&&&&\\
&&&&&&&&&&&&&\\
&&&&&&&&&&&&&\\
&&&&&&&&&&&&&\\
*{}\ar@{.}[uuuu]|{R}&&&
*{}\ar@{}[uuurrr]|{\smash{\pmb {\u A_1}}}
\ar@{.}[lll]
&&&*{}
\ar@{-}'[rruuuu]'[rruuuuuuuu]
&&&&&&&\\
  }
\hspace{-10pt}\xymatrix@=10pt{
\\ \\ \\
\too\\
}
\xymatrix@R=0pt@C=7pt
{
&&&&&&*{}\ar@{.}[d]&&&*{}&&&*{}&\\
&&&&&&*{}&&&
*{}\ar@{.}'[u]'[urrr][urrrd]
&&&*{}&\\
&&&&&&&&&&&&&\\
&&&&&&&&&&&&&\\
&&&&*{}\ar@{.}'[d][drrr]
\ar@{-}[uuuurr]
\ar@{-}[uuuurrrrr]|{\,I_t}
\ar@{-}'[rrr]'[rrrrruuuu]'[rrrrruuuulll]
&&&*{}&&&*{}\ar@{.}'[d]'&&&\\
&&&&&&&
*{}\ar@{-}'[rrr]'[rrrrruuuu]'[rrrrruuuulll]
'[rrrrruuuullllldddd]
\ar@{.}'[u]'[urrr]'[urrruuuurr]
&&&*{}\ar@{-}'[dddd]'[rr]'[rruuuu]
&&*{}&\\
&&&&&&&&&&&&&\\
&&&&&&&&&&&&&\\
&&*{}\ar@{.}'[d][drrr]
\ar@{--}[rruuuu]&&&*{}\ar@{--}[rruuuu]
&&&*{}&&&&&\\
&&&&&*{}\ar@{--}'[rruuuu]'
*{}\ar@{.}'[u]'[urrr]'[urrruuuurr]
&&*{}\ar@{-}[uuuu]
\ar@{}[uuuurr]|{\smash{\pmb{\u B_p}}}
&*{}\ar@{--}'[rruuuu]'
\ar@{-}[rr]&&*{}
&&&\\
&&&&&&&&&&&&&\\
&&&&&&&&&&&&&\\
*{}
\ar@{-}[uuuurrrrr]|{\,I_t}
\ar@{-}[rrr]
\ar@{-}[uuuurr]
&&&*{}\ar@{-}'[uuuurr]
'[uuuul]
&&&*{}\ar@{.}[lll]
\ar@{.}'[rruuuu]'[rruuu]
&&&&&&&\\
*{}\ar@{.}[u]
\ar@{.}'[rrrrrruuuuuuuuuuuu]
'[rrrrrrrrrruuuuuuuuuuuu]
\ar@{.}[rrr]
&&&*{}
\ar@{-}'[rruuuu]'[rruuuurrr]
'[rruuuurrrddddll]'[]'[dddd]
'[ddddrrr]'[rrr]
\ar@{.}[u]
&&&*{}\ar@{.}[u]&*{}&*{}\ar@{--}'[rruuuu]'&&&&&\\
&&&&&&&&&&&&&\\
&&&&&&&&&&&&&\\
&&&&&&&&&&&&&\\
*{}\ar@{.}[uuuu]&&&
*{}\ar@{}[uuurrr]|{\smash{\pmb{\u B_1}}}
\ar@{.}[lll]
&&&*{}
\ar@{-}'[rruuuu]'[rruuuuuuuu]
&&&&&&&\\
  }
  \]
Equating the spatial blocks
$(1,1,1),\dots,(1,1,p)$, we get
$S_1^TI_tU_1r=I_t$, \dots,
$S_1^TI_tU_pr=I_t$ (which follows from \eqref{keje} in analogy to
\eqref{krw}). Hence, $U_1=\dots=U_p$
and so $(R,S_2,U_1):\u A_1\too\u B_1,\
\dots,\ \u A_p\too\u B_p$.
\end{proof}

In the next two corollaries, we consider two important special cases of Theorem \ref{ktw1}.  Their proofs may help to understand the proof of Theorem \ref{ktw1}. If $(R,S,U): \u A \too \u B$, then we write that \emph{$\u A$ is reduced by $(R,S,U)$-transformations}. Recall that the problems of classifying $(1,2)$-tensors and $(0,3)$-tensors are the problems of classifying $3$-dimensional arrays up to $(S,S,S^{-T})$-transformations and $(S,S,S)$-transformations, respectively.
Recall also that each partitioned three-dimensional array is partitioned
into strata by
frontal, lateral, and horizontal
planes, and each stratum consists of slices; see  Section \ref{lku}. By a \emph{plane stratum}, we mean a stratum consisting of one slice.

\begin{corollary}\label{dek}
The problem \eqref{zzzz} contains the problem of classifying $(1,2)$-tensors.
\end{corollary}

\begin{proof}
Due to Theorem \ref{hep}, it suffices to prove that the second problem is contained in the problem of classifying partitioned three-dimensional
arrays up to block-equivalence.

For each unpartitioned array $\u A$ of size $n\times n\times n$, define the partitioned array
\begin{equation}\label{lje}
\begin{split}
 \u K^A=[\u K^A_{\alpha \beta\gamma }]_{\alpha =1}^{2}{}_{\beta =1}^{2}{}_{\gamma=1}^{1}:=
\quad\xymatrix@=0pt
{
&&&*{}&&&&&&*{}&&&&*{}\\
\\
\\
*{}
\ar@{-}'[rrruuu][rrrrrrrrruuu]
&&&&&&*{}&&&&*{}
\ar@{-}[rrruuu]
\ar@{-}[dddddd]
\ar@{-}[rrrddd]|{I_n\vphantom{A^A_A}}\\
\\
\\
&&&&&&&&&*{}&&&&*{}\\
\\
\\
*{}
\ar@{}[rrrrrruuuuuu]|{\underline {\pmb A}}
\ar@{-}'[rrrrrr]'[rrrrrruuuuuu]
'[uuuuuu][]
&&&&&&*{}
\ar@{-}'[rrruuu]'[rrruuuuuuuuu][uuuuuu]
&&&&*{}
\ar@{-}'[rrruuu][rrruuuuuuuuu]
\\
&&&*{}&&&&&&*{}\\
\\
\\
*{}\ar@{-}[rrrrrr]
\ar@{-}[rrruuu]
\ar@{-}[rrrrrrrrruuu]|{\,I_n\vphantom{A^A_A}}
&&&&&&*{}
\ar@{-}'[uuurrr][uuulll]
\\
}
\end{split}
\end{equation}
of size $(n+1)\times (n+1)\times n$. It is obtained from $\u A$ by attaching under it and on the right of it the plane strata that are the identity matrices:
\[
\u K^A_{111}=\u A,\quad \u K^A_{121}=I_n
,\quad \u K^A_{211}=I_n,\quad \u K^A_{221}=0
\]
(the diagonal lines in $\u K^A_{121}$
and $\u K^A_{211}$ denote the main
diagonal of $I_n$ consisting of units).
Let $\u B$ be another unpartitioned
array of size $n\times n\times n$, and
let $\u K^{A}$ and $\u K^{B}$
be block-equivalent. This means that there exists
\[(R\oplus [r],S\oplus [s],U): \u K^{A} \too \u K^{B};\] that is,
\[
\xymatrix@=0pt
{
&&&*{}&&&&&&*{}&&&&*{}\\
\\
\\
*{}
\ar@{-}'[rrruuu][rrrrrrrrruuu]
&&&&&&*{}&&&&*{}
\ar@{-}[rrruuu]|{\vphantom{A^A_A}U\vphantom{A^A_A}}
\ar@{-}[dddddd]|{R\vphantom{A^A_A}}
\ar@{-}[rrrddd]|{\vphantom{A^A_A}I_n\vphantom{A^A_A}}\\
\\
\\
&&&&&&&&&*{}&&&&*{}\\
\\
\\
*{}
\ar@{}[rrrrrruuuuuu]|{\underline {\pmb A}}
\ar@{-}'[rrrrrr]'[rrrrrruuuuuu]
'[uuuuuu][]
&&&&&&*{}
\ar@{-}'[rrruuu]'[rrruuuuuuuuu][uuuuuu]
&&&&*{}
\ar@{-}'[rrruuu][rrruuuuuuuuu]
\\
&&&*{}&&&&&&*{}\\
\\
\\
*{}\ar@{-}[rrrrrr]|{S}
\ar@{-}[rrruuu]|{U\vphantom{A^A_A}}
\ar@{-}[rrrrrrrrruuu]|
{\,I_n\vphantom{A^A_A}}
&&&&&&*{}
\ar@{-}'[uuurrr][uuulll]
\\
}
                                  \quad
\xymatrix@=10pt{
\\ \\
\too\\
}
                     \ \
\xymatrix@=0pt
{
&&&*{}&&&&&&*{}&&&&*{}\\
\\
\\
*{}
\ar@{-}'[rrruuu][rrrrrrrrruuu]
&&&&&&*{}&&&&*{}
\ar@{-}[rrruuu]
\ar@{-}[dddddd]
\ar@{-}[rrrddd]|{\vphantom{A^A_A}I_n\vphantom{A^A_A}}\\
\\
\\
&&&&&&&&&*{}&&&&*{}\\
\\
\\
*{}
\ar@{}[rrrrrruuuuuu]|{\underline {\pmb B}}
\ar@{-}'[rrrrrr]'[rrrrrruuuuuu]
'[uuuuuu][]
&&&&&&*{}
\ar@{-}'[rrruuu]'[rrruuuuuuuuu][uuuuuu]
&&&&*{}
\ar@{-}'[rrruuu][rrruuuuuuuuu]
\\
&&&*{}&&&&&&*{}\\
\\
\\
*{}\ar@{-}[rrrrrr]
\ar@{-}[rrruuu]
\ar@{-}[rrrrrrrrruuu]|{\,I_n\vphantom{A^A_A}}
&&&&&&*{}
\ar@{-}'[uuurrr][uuulll]
\\
}
\]
Equating the spatial blocks
$(1,2,1)$ and
$(2,1,1)$, we get $R^TI_nUs=I_n$ and $S^TI_nUr=I_n$. Hence $U^{-T}=Rs=Sr$, and so
$(R,S,U)=(R,Rsr^{-1},R^{-T}s^{-1}):\u A\too \u B$. Thus, $(R,Rr^{-1},R^{-T}):\u A\too \u B$, and so $(Rr^{-1},Rr^{-1},(Rr^{-1})^{-T}):\u A\too \u B$.
\end{proof}

\begin{corollary}\label{dek1}
The problem \eqref{zzzz} contains the problem of classifying $(0,3)$-tensors.
\end{corollary}

\begin{proof}
For each unpartitioned array $\u A$ of
size $n\times n\times n$, define the
array
\[
\u L^A:=\quad
\xymatrix@R=7pt@C=7pt{
*{}&&*{}&&*{}&&*{}&\\
&*{}\ar@{}[rrru]|{{\underline {\pmb Q}}}&&&&*{}\ar@{.}'[ru][ul]&\\
*{}&&*{}
&&*{}\ar@{-}[ur]&*{}&*{}&{\smash{\quad\text{with $\u P=\u Q=0$}}}\\
&&&&&*{}\ar@{.}'[ur][uuur]&&
&\\
*{}\ar@{}[rruu]|{\underline {\pmb A}}
\ar@{-}'[rrrr]'[rrrruu]'[uu][]
&&\ar@{}[rruu]|{\underline {\pmb P}}
\ar@{-}'[uu]'[uuuurr]'[uuuu][uull]
&&*{}\ar@{-}'[ru]'[ruuu][uuulll]&\\
}\]
of size $n\times 2n\times 2n$
partitioned into $1\cdot 2\cdot 2$
spatial blocks.
We attach the plane strata to the right of $\u L^A$, under it, and behind it. All blocks of the new plane strata are zero except for the identity
matrices $I_n$ under and directly behind of
$\u P$ and under and on the right of
$\u Q$, and also except for the matrix $I_1$ at the intersection of these planes:
\begin{equation}\label{gfl}
\begin{split}
\u N^A:=\xymatrix@R=7pt@C=14pt{
*{}&&*{}&&*{}&*{}&*{}&*{}
\ar@{-}'[dd]&&*{}
           \\
*{}&&*{}&&*{}&&*{}&&&*{}
           \\
*{}&&&*{}
&&*{}\ar@{-}'[rrrr]'[rrrruu]'[uu][]&&*{}
\ar@{-}[rruu]|{I_n}&&*{}\ar@{.}[rrr]&&&{I_1}
           \\
*{}&&*{}&&*{}&&*{}&&&*{}
           \\
&*{}\ar@{}[rrru]|{{\underline {\pmb Q}}}&&&&*{}\ar@{-}'[ru][ul]&&&*{}
\ar@{-}[rd]|{I_n}&
         \\
*{}&&*{}
&&*{}\ar@{-}[ur]&*{}&*{}&*{}&&*{}
\ar@{.}[uuurrr]
         \\
&&&&&*{}\ar@{-}'[ur][uuur]&&&*{}\ar@{-}[uu]&
        \\
*{}\ar@{}[rruu]|{\underline {\pmb A}}
\ar@{-}'[rrrr]'[rrrruu]'[uu][]
&&*{}
\ar@{}[rruu]|{\underline {\pmb P}}
\ar@{-}'[uu]'[uuuurr]'[uuuu][uull]
&&*{}\ar@{-}'[ru]'[ruuu][uuulll]
&&&*{}
\ar@{-}'[uu]'[uuuurr]'[uurr][]
&&
         \\
*{}&&*{}&&*{}&&*{}
           \\
&*{}\ar@{-}[rrru]|{I_n}\ar@{-}[rrrr]&&&&*{}
         \\
*{}&&*{}\ar@{-}[rruu]\ar@{-}[rrru]|{I_n}
&&*{}\ar@{-}'[uurr]'[uull]'[llll][]&*{}&*{}
}
 \end{split}
\end{equation}
The obtained array $\u N^A$ is partitioned into $2\cdot 3\cdot 3$ spatial blocks.

Let $\u B$ be another unpartitioned
array of size $n\times n\times n$. Let $\u N^{A}$ and $\u N^{B}$
be block-equivalent:
\[
(R\oplus[r],S_1\oplus S_2\oplus [s],
U_1\oplus U_2\oplus [u]): \u N^{A} \too \u N^{B}.
\]
Equating the spatial blocks with $I_n$ and $I_1$, we get
\[
S_1^TU_2r=I_n,\quad S_2^TU_1r=I_n,\quad R^TU_2s=I_n,\quad R^TS_2u=I_n,\quad rsu=1.
\]
Hence
\[
S_1=U_2^{-T}r^{-1}=Rr^{-1}s
, \qquad
U_1=S_2^{-T}r^{-1}=Rr^{-1}u
\]
and we have
\[
(R,S_1,U_1)=(Rr^{-1}r,Rr^{-1}s,Rr^{-1}u):\u{A} \too \u{B}.
\]
Since $rsu=1$, $(Rr^{-1},Rr^{-1},Rr^{-1}):\u{A} \too \u{B}$.
\end{proof}

\section{Proof of the statement (\ref{kjj})} \label{sss4}

Each block-equivalence transformation of a partitioned three-dimensional array
$ \u A$
has the form
\begin{equation}\label{dsy1}
(R,S,U)=(R_1\oplus\dots\oplus R_{\o m},\
S_1\oplus\dots\oplus S_{\o n},\
U_1\oplus\dots\oplus U_{\o t}):\u A \too\u B.
\end{equation}
In this section, we consider a special case of these transformations: some of the diagonal blocks in $R,S,U$ are claimed to be equal or mutually contragredient.

Let us give formal definitions. Consider a finite set $\mathcal P$ with two relations $\sim$ and $\Join$ satisfying the following
conditions:
\begin{itemize}
\item[(i)] $\sim$ is an equivalence relation, \label{page}
  \item[(ii)] if $a\Join b$, then $a\nsim b$,
  \item[(iii)] if $a\Join b$, then $b\Join c$ if and only if $a\sim c$.
\end{itemize}
Taking $a=c$ in (iii), we obtain that $a\Join b$ implies $b\Join a$.

It is clear that the relation  $\Join$ can be extended to the set $\mathcal P\!/\!\!\sim$ of equivalence classes such that $a\Join b$ if and only if $[a]\Join [b]$, where $[a]$ and $[b]$ are the equivalence classes of $a$ and $b$.
Moreover, if an equivalence relation $\sim$ on $\mathcal P$ is fixed and $*$ is any involutive mapping on $\mathcal P\!/\!\!\sim$ (i.e., $[a]^{**}=[a]$ for each $[a]\in\mathcal P\!/\!\!\sim$), then the relation $\Join$ defined on $\mathcal P$ as follows:
\[
a\Join b\quad\Longleftrightarrow\quad [a]\ne [a]^*=[b]
\]
satisfies (ii) and (iii), and each relation $\Join$ satisfying (ii) and (iii) can be so obtained.

Let
\begin{equation}\label{jdp}
 \mathcal P:=\{1,\dots,\bar m;1',\dots,\bar n';1'',\dots,\bar t''\}
\end{equation}
be the disjoint union of the set of first indices, the set of second indices, and the set of third indices of $ \u A=[\u
A_{\alpha \beta \gamma }]_{\alpha
=1}^{\o m}\,{}_{\beta=1}^{\o
n}\,{}_{\gamma=1}^{\o t}$. Since these sets correspond to nonintersecting subsets of $\mathcal P$, we can denote all transforming matrices in \eqref{dsy1} by the same letter:
\begin{equation*}\label{mmu}
(S_1\oplus\dots\oplus S_{\o m},\ S_{1'}\oplus \dots\oplus S_{\o n'},\
S_{1''}\oplus\dots\oplus S_{\o t''}):\u A \too\u B,
\end{equation*}
and give the partition of $\u A$ by the sequence
\begin{equation}\label{hge}
\u d:=(d_1,\dots,d_{\o m};\ d_{1'},\dots,d_{\o n'};\ d_{1''},\dots,d_{\o t''})
\end{equation}
(in which the semicolons separate the sets of sizes of
frontal, lateral, and horizontal strata)
such that the size of each $\u A_{ijk}$ is $d_i\times d_{j'}\times d_{k''}$.

Let $\sim$ and $\Join$ be binary relations on \eqref{jdp} satisfying (i)--(iii). Let the partition \eqref{hge} of $ \u A=[\u
A_{\alpha \beta \gamma }]_{\alpha
=1}^{\o m}\,{}_{\beta=1}^{\o
n}\,{}_{\gamma=1}^{\o t}$ satisfies the condition:
\begin{equation}\label{chp}
\text{$d_{\alpha }=d_{\beta }$ if $\alpha \sim\beta $ or $\alpha \Join\beta $\qquad for all $\alpha,\beta \in\cal P$.}
\end{equation}
We say that \eqref{jdp} is a \emph{linked block-equivalence transformation} if the following two conditions hold for all $\alpha,\beta \in\cal P$:
\begin{equation}\label{ghq}
\text{$S_{\alpha }=S_{\beta }$ if $\alpha \sim\beta $,\qquad
$S_{\alpha }=S_{\beta }^{-T}$ if $\alpha \Join\beta $.}
\end{equation}

It is convenient to give the relations $\sim$ and $\Join$ on $\cal P$ by
\begin{equation}
\label{gra}
\parbox[c]{0.8\textwidth}{the graph $Q$ with the set of vertices $\cal P$ and with two types of arrows: two vertices $\alpha $ and $\beta $ are linked by a solid line if $\alpha \sim\beta $ and $\alpha \ne\beta $, and by a dotted line if $\alpha \Join\beta $.
}
\end{equation}
\begin{example}
The graphs
 
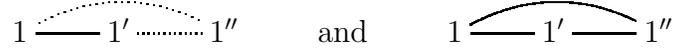
\[
\xymatrix{
1\ar@{-}[r]\ar@/^1pc/@{.}[rr]
&1'\ar@{.}[r]&1''&\text{and}&
1\ar@{-}[r]\ar@/^1pc/@{-}[rr]
&1'\ar@{-}[r]&1''
\\
}
\]
give the problems of classifying partitioned arrays consisting of a single spatial block $\u A_{111}$ up to $(S,S,S^{-T})$-transformations and up to $(S,S,S)$-transformations, respectively; that is, the problems of classifying $(1,2)$-tensors and $(0,3)$-tensors.
\end{example}

\begin{example}\label{bbn}
The graph
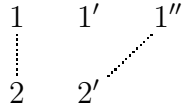
\[
\xymatrix@R=14pt@C=14pt{
1\ar@{.}[d]&1'&1''\ar@{.}[ld]\\
2&2'
\\}
\]
gives the problem of classifying partitioned arrays $\u A=[\u A_{\alpha \beta 1}]$ consisting of $2\cdot 2\cdot 1$ spatial blocks up to $(R\oplus R^{-T},S\oplus U^{-T},U)$-transformations; that is, up to linked block-equivalence transformations
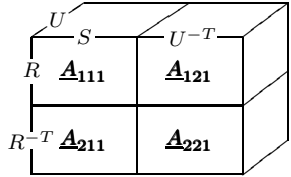
\[
\xymatrix@R=7pt@C=14pt{
&*{}\ar@{-}[rrrr]&&*{}&&*{}\\
*{}\ar@{-}[ur]|{\,U\,}
\ar@{-}[rr]|{\,S\,}
&&*{}
\ar@{-}[rr]|{\;U^{-T}}&&*{}\\
&&&&&*{}\\
*{}\ar@{}[rruu]|{\pmb{\u A_{111}}}\ar@{-}[uu]|{R\vphantom{A^A_A}}
\ar@{-}'[rrrr][rrrrru]&&*{}
\ar@{}[rruu]|{\pmb{\u A_{121}}}
&&*{}\\
&&&&&*{}\\
*{}\ar@{}[rruu]|{\pmb{\u A_{211}}}
\ar@{-}[uu]|{R^{-T}}\ar@{-}'[rrrr]'[rrrruuuu]
&&*{}\ar@{}[rruu]|{\pmb{\u A_{221}}}
\ar@{-}'[uuuu][uuuuur]&&*{}\ar@{-}'[ru]'[ruuuuu][uuuu]\\
}
 \]
This problem can be considered as the problem of classifying array-representations
$\mathcal A=(\u A_{111},\u A_{121},\u A_{211},\u A_{221})$ of the directed bipartite
graph
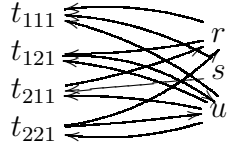
\[
\xymatrix@R=-6pt@C=10ex{
t_{111}
  \ar@/^/@{<-}[rd]
  \ar@/^/@{<-}[rddd]
  \ar@/^/@{<-}[rddddd]\\
&{r\quad}
\\
t_{121}
  \ar@/_3pt/@{<-}[ru]
  \ar@<-0.5ex>@/^9pt/@{->}[rddd]
  \ar@<-0.5ex>@/^5pt/@{<-}[rddd]\\
&{s\quad }
\\
t_{211}
  \ar@/_3pt/@{->}[ruuu]
  \ar@<-0.3ex>@{<-}[ru]
  \ar@<-0.5ex>@/^3pt/@{<-}[rd]
\\
&{u\quad }
\\
t_{221}
  \ar@/_9pt/@{->}[ruuuuu]
  \ar@/_1pt/@{->}[ru]
  \ar@<-0.1ex>@/_/@{<-}[ru]
\\
}
\]
\end{example}

We prove the statement \eqref{kjj} in the following theorem. Its array $\u H({\u X})$ generalizes the arrays \eqref{lje} and \eqref{gfl}.

\begin{theorem}\label{kts}
Let $ \mathcal P:=\{1,\dots,\bar m;1',\dots,\bar n';1'',\dots,\bar t''\}
$ be a set with binary relations $\sim$ and $\Join$ satisfying the conditions {\rm(i)--(iii)} from the beginning of Section \ref{sss4}. Let $ \u X=[\u
X_{\alpha \beta \gamma }]_{\alpha
=1}^{\o m}\,{}_{\beta=1}^{\o
n}\,{}_{\gamma=1}^{\o t}$ be a variable array whose entries are independent parameters and
whose partition into spatial blocks is given by some sequence \eqref{hge} satisfying \eqref{chp}.

Then there exists a partitioned array $\u H({\u X})=[\u H_{\alpha \beta \gamma }]_{\alpha
=1}^{\Bar{\Bar m}}\,{}_{\beta=1}^{\Bar{\Bar
n}}\,{}_{\gamma=1}^{\Bar{\Bar t}}$ in which
 $\o m< \Bar{\Bar m}$, $\o n< \Bar{\Bar n}$, $\o t< \Bar{\Bar t}$;
\begin{itemize}
  \item[\rm(a)] $\u X$ is the subarray of $\u H({\u X})$ located at the first $\o m\cdot\o n\cdot\o t$ spatial blocks  (i.e.
  $\u H_{\alpha \beta \gamma }=\u X_{\alpha \beta \gamma }$ if $\alpha\le\o m$, $\beta\le\o n$, and $\gamma\le\o t$);

  \item[\rm(b)] $\u H_{\Bar{\Bar m}\Bar{\Bar n}\Bar{\Bar t}}=[1]$ is of size $1\times 1\times 1$, and the other spatial blocks outside of $\u X$ are zero except for some of $\u H_{\Bar{\Bar m} \beta \gamma}$, $\u H_{\alpha \Bar{\Bar n} \gamma}$, $\u H_{\alpha \beta \Bar{\Bar t}}$
      that are the identity matrices
\end{itemize}
such that two three-dimensional arrays $\u
A$ and $\u B$ over a field, partitioned conformally to $\u X$, are linked block-equivalent if and only if
$\u H({\u A})$ and $\u H({\u B})$ are block-equivalent.
Note that the disposition of the identity matrices outside of $\u X$ (see {\rm(b)}) depends only on  $ (\mathcal P,\sim,\Join)$.
\end{theorem}

  \begin{proof}
Let $\Bar{\Bar m}>\o m$, $\Bar{\Bar n}>\o n$, and $\Bar{\Bar t}>\o t$ be large enough in order to make possible the further arguments. Let $\u X$ be the variable array from the theorem.  Denote by $\u H^0({\u X})=[\u H^0_{\alpha \beta \gamma }]_{\alpha
=1}^{\Bar{\Bar m}}\,{}_{\beta=1}^{\Bar{\Bar
n}}\,{}_{\gamma=1}^{\Bar{\Bar t}}$ the array that satisfies (a) and in which all the spatial blocks outside of $\u X$ are zero except for $\u H^0_{\Bar{\Bar m}\Bar{\Bar n}\Bar{\Bar t}}=[1]$ of size $1\times 1\times 1$.

Let $Q$ be the graph defined in \eqref{gra} that gives the relations $\sim$ and $\Join$ from the theorem. Consider the graph
\begin{equation}\label{kkp}
\begin{split}
\xymatrix@R=0pt@C=7pt{
&1&1'&1''\\\\
 &2&2'&2''\\
\smash{Q^0:\quad } \\
&\vdots&\vdots&\vdots\\ \\
&\Bar{\Bar m}&\Bar{\Bar n}&\Bar{\Bar t}
\\
}
\end{split}
\end{equation}
without edges, whose vertices correspond to the indices of $\u H^0({\u X})$.
Each vertex of $Q$ is the vertex of $Q^0$.
We will consecutively join the edges of $Q$ to $Q^0$ and respectively modify $\u H^0({\u X})$ until obtain $Q^r\supset Q$ and $\u H({\u X}):=\u H^r({\u X})$ satisfying Theorem \ref{kts}.

On each step $k$, we construct
$Q^k$ and $\u H^k({\u X})=[\u H^k_{\alpha \beta \gamma }]$ such that
\begin{itemize}
\item[($1^{\circ}$)]
the conditions (a) and (b) hold with $\u H^k({\u X})$ instead of $\u H({\u X})$, and

\item[($2^{\circ}$)]
for every two arrays $\u A$ and $\u B$ partitioned  conformally to $\u X$, each block-equivalence
\begin{equation}\label{khp}
(S_{1}\oplus\dots\oplus S_{\Bar{\Bar m}},
S_{1'}\oplus\dots\oplus S_{\Bar{\Bar n}'},
S_{1''}\oplus\dots\oplus S_{\Bar{\Bar t}''}): \u H^{k}({\u A})\too \u H^{k}({\u B})
\end{equation}
with $S_{\Bar{\Bar m}}=
S_{\Bar{\Bar n}'}=S_{\Bar{\Bar t}''}=[1]$
satisfies \eqref{ghq} in which the relations $\sim$ and $\Join$ are given by
the graph $Q^k$.
\end{itemize}

If $k=0$, then ($1^{\circ}$) and ($2^{\circ}$) hold since the block-equivalence coincides with the linked block-equivalence with respect to the relations $\sim$ and $\Join$ given by $Q^0$.

Reasoning by induction, we assume that
$Q^k$ and $\u H^k({\u X})$ satisfying ($1^{\circ}$) and ($2^{\circ}$) have been constructed. We construct $Q^{k+1}$ and $\u H^{k+1}({\u X})=[\u H^{k+1}_{\alpha \beta \gamma }]$ as follows.
Let $\lambda $ be  an edge  of $Q$ that does not belong to $Q^k$. Denote by $Q^{k+1}$ the graph obtained from $Q^k$ by joining $\lambda $ and all the edges that appear automatically due to the transitivity of $\sim$ and the condition (iii) from the beginning of this section.

For definiteness, we suppose that $\lambda $ connects a vertex from the first column and a vertex from the first or second column in \eqref{kkp}. The following cases are possible.
\begin{description}
  \item[Case 1:] \emph{$\lambda $  is dotted and connects a vertex from the first column and a vertex from the second column in \eqref{kkp}.}
Let $\lambda :\xymatrix@C=14pt@1{\alpha \ar@{.}[r] & \beta'}$. We replace the spatial block $\u H^k_{\alpha \beta \Bar{\Bar t} }=0$ of size $d_{\alpha }\times d_{\beta '}\times 1$ by the identity matrix: $\u H^{k+1}_{\alpha \beta \Bar{\Bar t }}=I$ ($d_{\alpha }=d_{\beta '}$ by \eqref{chp}); the other spatial blocks of $\u H^k$ and $\u H^{k+1}$ coincide.
Since $\u H^{k+1}_{\alpha \beta \Bar{\Bar t} }=I$ and $S_{\Bar{\Bar t}''}=[1]$, we have $S_{\alpha}^TIS_{\beta'}=I$ in \eqref{khp}. Hence, $S_{\alpha}^{-T}=S_{\beta'}$, which ensures the conditions ($1^{\circ}$) and ($2^{\circ}$) with $k+1$ instead of $k$.

\item[Case 2:] \emph{$\lambda $ is solid and connects a vertex from the first column and a vertex from the second column in \eqref{kkp}.} Let $\lambda :\xymatrix@C=14pt@1{\alpha \ar@{-}[r] & \beta'}$ and let $\gamma ''\in\{\bar t+1,\dots,\Bar{\Bar t}-1\}$ be a vertex from the third column in \eqref{kkp} that does not have arrows (it exists since we have supposed that $\Bar{\Bar m},\Bar{\Bar n},\Bar{\Bar t}$ are large enough). Reasoning as in Case 1, we join the following two dotted arrows to $Q^k$:
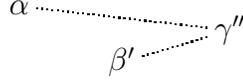
\[
\xymatrix@R=-5pt{
{\alpha }\ar@{.}[rrd]\\
&&{\gamma''}\\
&{\beta'}\ar@{.}[ru]
\\
}
\]
Then the solid arrow $\lambda :\xymatrix@C=14pt@1{\alpha \ar@{-}[r] & \beta'}$ is joined automatically by (iii).

\item[Case 3:] \emph{$\lambda $ is solid and connects two vertices from the first column.} Let $\lambda :\xymatrix@C=14pt@1{\alpha \ar@{-}[r] & \beta}$. Reasoning as in Case 2, we join the following two dotted arrows to $Q^k$:
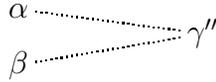
\[
\xymatrix@R=-5pt{
{\alpha }\ar@{.}[rrd]\\
&&{\gamma''}\\
{\beta}\ar@{.}[rru]
\\
}
\]

\item[Case 4:]
\emph{$\lambda $ is dotted and connects two vertices from the first column.} Let $\lambda :\xymatrix@C=14pt@1{\alpha \ar@{.}[r] & \beta}$. Let $\gamma '\in\{\bar n+1,\dots,\Bar{\Bar n}-1\}$ and $\delta''\in\{\bar t+1,\dots,\Bar{\Bar t}-1\}$ be vertices from the second and third columns in \eqref{kkp} that do not have arrows. We join the following  three dotted arrows to $Q^k$:
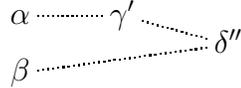
\[
\xymatrix@R=-5pt{
{\alpha }\ar@{.}[r]&{\gamma '}\ar@{.}[rd]\\
&&{\delta''}\\
{\beta}\ar@{.}[rru]
\\
}
\]
Then the dotted arrow $\lambda :\xymatrix@C=14pt@1{\alpha \ar@{.}[r] & \beta}$ is attached automatically by (iii).
\end{description}

The conditions ($1^{\circ}$) and ($2^{\circ}$) with $k+1$ instead of $k$ hold in all the cases. We repeat this construction until obtain $Q^r\supset Q$.

Let $\u A$ and $\u B$ be three-dimensional arrays partitioned  conformally to $\u X$ such that
$\u H^r({\u A})$ and $\u H^r({\u B})$ are block-equivalent; that is, there exists
\[
(R_{1}\oplus\dots\oplus R_{\Bar{\Bar m}},
R_{1'}\oplus\dots\oplus R_{\Bar{\Bar n}'},
R_{1''}\oplus\dots\oplus R_{\Bar{\Bar t}''}): \u H^r({\u A})\too \u H^r({\u B}).
\]
Since $\u H^r_{\Bar{\Bar m}\Bar{\Bar n}\Bar{\Bar t}}$ is of size $1\times 1\times 1$, the summands $R_{\Bar{\Bar m}}, R_{\Bar{\Bar n}'}, R_{\Bar{\Bar t}''}$ are $1\times 1$. Let $R_{\Bar{\Bar m}}=[a]$, $R_{\Bar{\Bar n}'}=[b]$, and $R_{\Bar{\Bar t}''}=[c]$. Since $\u H^r_{\Bar{\Bar m}\Bar{\Bar n}\Bar{\Bar t}}=[1]$, $abc=1$. Hence,
\[
(a^{-1}(R_{1}\oplus\dots\oplus R_{\Bar{\Bar m}}),
b^{-1}(R_{1'}\oplus\dots\oplus R_{\Bar{\Bar n}'}),
c^{-1}(R_{1''}\oplus\dots\oplus R_{\Bar{\Bar t}''})): \u H^r({\u A})\too \u H^r({\u B}).
\]

Since $a^{-1}R_{\Bar{\Bar m}}=b^{-1}R_{\Bar{\Bar n}'}=c^{-1}R_{\Bar{\Bar t}''}=[1]$, we use ($1^{\circ}$) and ($2^{\circ}$) with $k=r$ and get that
\[
(a^{-1}(R_{1}\oplus\dots\oplus R_{\Bar m}),
b^{-1}(R_{1'}\oplus\dots\oplus R_{\Bar n'}),
c^{-1}(R_{1''}\oplus\dots\oplus R_{\Bar t''})): \u A\too \u B
\]
is a linked block-equivalence with respect to the relations $\sim$ and $\Join$ from Theorem \ref{kts}.
\end{proof}


\section{Proof of Theorem \ref{ktw1}}
\label{sss5}

In this section, we prove  Theorem \ref{ktw1}, which ensures  the statement \eqref{kjt3}.
\begin{lemma}\label{due}
It suffices to prove Theorem \ref{ktw1} for all graphs $G$ in which each left vertex has three arrows.
\end{lemma}

\begin{proof}
Let $G$ be a directed bipartite
graph with ${\T}=\{t_1,\dots,t_p\}$
and $\V=\{1,\dots,q\}$.

\begin{itemize}
\item[($1^{\circ}$)]
Let $G$ have a left vertex $t$ with exactly two arrows:
\begin{equation}\label{fko}
\begin{split}
\xymatrix@R=0pt@C=50pt{
&u\\
 {t}\ar@{-}[r]
 \ar@{-}@<0.2ex>[ru]&v&
 \smash{\text{\raisebox{5pt}{or}}}
 &{t}\ar@{-}@/^10pt/[r]
 \ar@{-}[r]&{v}
}
\end{split}
\end{equation}
where each line is $\longrightarrow$ or
$\longleftarrow$. Denote by $G'$ the
graph obtained from $G$ by replacing
\eqref{fko} with
\begin{equation}\label{ddi}
\begin{split}
\xymatrix@R=0pt@C=30pt{
&&{u\quad }\\
 {\quad t}\ar@{-}[rr]
 \ar@{-}@<0.2ex>[rru]&&{v\quad }&
\smash{\text{\raisebox{-4pt}{or}}}&{\quad t}\ar@{-}@/^10pt/[rr]
 \ar@{-}[rr]&&{v \quad }
\\
{t_{\smash{p+1}}}\ar@{->}@/_14pt/[rr]
  &&{q\smash{+1}}\ar@{->}@/^/[ll]
  \ar[ll]\ar[llu]&&
{t_{\smash{p+1}}}\ar@{->}@/_14pt/[rr]
  &&{q\smash{+1}}\ar@{->}@/^/[ll]
  \ar[ll]\ar[llu]
}
\end{split}
\end{equation}
\\respectively (thus, $\T'=\{t_1,t_2,\dots,t_{p+1}\}$ and $\V'=\{1,2,\dots,q+1\}$). Let us extend each array-representation $\cal A$ of $G$ to the array-representation ${\cal A}'$ of $G'$ by assigning the $1\times 1\times 1$ array $\u A_{t_{p+1}}':=[1]$  to the new vertex $t_{p+1}$.

Let us prove that ${\cal A}\simeq{\cal
B}$ if and only if ${\cal
A}'\simeq{\cal B}'$. If
$(S_1,\dots,S_q):{\cal A}\too{\cal B}$,
then $(S_1,\dots,S_q,[1]):{\cal
A}'\too{\cal B}'$. Conversely, let
$(S_1,\dots,S_q,[a]):{\cal A}'\too{\cal
B}'$. Then $([a],[a],[a]^{-T}):\u
A_{t_{p+1}}\too\u B_{t_{p+1}}$. Since
$\u A_{t_{p+1}}=\u B_{t_{p+1}}=[1]$, we
have $a=1$, and so
$(S_1,\dots,S_q):{\cal A}\too{\cal B}$.

\item[($2^{\circ}$)]
Let $G$ have a left vertex $t$ with
exactly one arrow: $t\longrightarrow v$
or $t\longleftarrow v$. Then by analogy
with \eqref{ddi} we replace it with
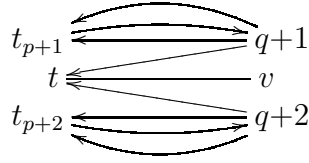
\[
\xymatrix@R=0pt@C=30pt{
{t_{\smash{p+1}}}\ar@{<-}@/^14pt/[rr]
  &&{q\smash{+1}}\ar@{<-}@/_/[ll]
  \ar[ll]\ar[lld]
\\
 {\quad t}\ar@{-}[rr]&&{v\quad}
\\
{t_{\smash{p+2}}}\ar@{<-}@/_14pt/[rr]
  &&{q\smash{+2}}\ar@{<-}@/^/[ll]
  \ar[ll]\ar[llu]
}
\]
in which $t\lin v$ is $t\longrightarrow
v$ or $t\longleftarrow v$.
\end{itemize}

We repeat ($1^{\circ}$) and ($2^{\circ}$) until
extend $G$ to a graph $\widetilde G$ with the set of left vertices $\widetilde{\T}=\{t_1,t_2,\dots,t_{p+p'}\}$ in
which each left vertex has exactly
three arrows. For
array-representations ${\cal A}=(\u
A_1,\dots,\u A_p)$ and ${\cal B}=(\u
B_1,\dots,\u B_p)$ of $G$, define
the array-representations
\begin{equation}\label{lic}
\widetilde{\cal
A}:=(\u A_1,\dots,\u A_p,[1],\dots,[1]),\qquad
\widetilde{\cal
B}:=(\u B_1,\dots,\u B_p,[1],\dots,[1])
\end{equation}
of $\widetilde G$ and obtain that
${\cal A}\simeq{\cal B}$ if and only if
$\widetilde{\cal A}\simeq\widetilde{\cal B}$.

Suppose that Theorem \ref{ktw1} holds for $\widetilde G$  and some
array $\u{\widetilde{F}}(\widetilde{\cal X})$,
in which $\widetilde{\cal X}=(\u X_1,\dots,\u
X_p,[x_{p+1}],\dots,[x_{p+p'}])$ is a
variable array-representation of $\widetilde G$. Substituting the array-representations \eqref{lic}, we find that
$\u{\widetilde{F}}(\widetilde{\cal A})$ is
equivalent to
$\u{\widetilde{F}}(\widetilde{\cal B})$ if and
only if $\widetilde{\cal
A}\simeq\widetilde{\cal B}$,  if and only if
${\cal A}\simeq{\cal B}$. Hence, Theorem \ref{ktw1} holds for $G$  and for the
array $\u{F}({\cal
X}):=\u{\widetilde{F}}(\u X_1,\dots,\u
X_p,[1],\dots,[1])$ with ${\cal X}:=(\u X_1,\dots,\u
X_p)$.
\end{proof}

The \emph{direct sum} $\u A\oplus \u B$ of three-dimensional arrays $\u A$ and $\u B$ is the partitioned array $[\u C_{\alpha \beta }]_{\alpha ,\beta =1}^2$, in which $\u C_{11}:=\u A,\  \u C_{22}:=\u B,$ and the other $\u C_{\alpha \beta }:=0$:
 \[
\u A\oplus\u B:=\xymatrix@R=3pt@C=5pt{
&*{}&&*{}\\
*{}&&*{}\ar@{-}[ru]&&*{}&&*{}
           \\
&&&*{}&&*{}\ar@{-}[ru]
          \\
*{}\ar@{}[rruu]|{\pmb{\u A}}
\ar@{-}'[rr]'[rruu]'[uu][]
&&*{}\ar@{-}'[ru]'[ruuu]'[luuu][lluu]&
&&&*{}
           \\
&&&*{}\ar@{}[rruu]|{\pmb{\u B}}
\ar@{-}'[rr]'[rruu]'[uu][]&&
*{}\ar@{-}'[ru]'[ruuu]'[luuu][lluu]&\\
}
\]

\begin{proof}[Proof of Theorem \ref{ktw1}]
Let $G$ be a directed bipartite graph with
left vertices $t_1,\dots,t_p$  and right
vertices  $1,\dots,q$. Due to
Lemma \ref{due}, we can suppose that
each left vertex of $G$ has exactly three arrows:
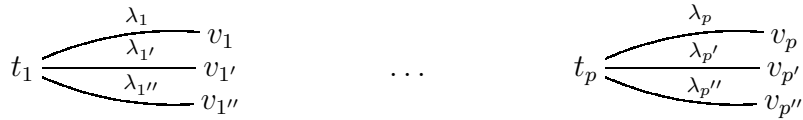
\[
\xymatrix@R=-3pt@C=2cm
{
&v_1&&&v_p\\
{t_1}\ar@/^/@{-}[ru]^(0.6){\lambda _1}
\ar@{-}[r]^(0.6){\lambda_{1'}}
\ar@/_/@{-}[rd]^(0.6){\lambda_{1''}}
&v_{1'}&
   \hspace{-1cm}  \dots \hspace{-1cm}
&{t_p}\ar@/^/@{-}[ru]^(0.6){\lambda_p}
\ar@{-}[r]^(0.6){\lambda_{\smash{p'}}}
\ar@/_/@{-}[rd]^(0.6){\lambda_{\smash{p''}}}
&v_{p'}\\
&v_{1''}&&&v_{p''}\\
}
\]

Let
\begin{equation}\label{ccc}
{\cal A}=(\u A_1,\dots,\u A_p),\qquad {\cal B}=(\u B_1,\dots,\u B_p)
\end{equation}
be two array-representations of $G$ of the same size. By Definition \ref{aal}, each isomorphism ${\cal S}:{\cal A}\too {\cal B}$ is given by a sequence
\begin{equation}\label{yyy}
{\cal S}:=(S_1,\dots,S_q)
\end{equation}
 of
nonsingular matrices such that
\begin{equation}\label{mbh}
(S_{v_i}^{\tau_i},S_{v_{i'}}^{\tau_{i'}}, S_{v_{i''}}^{\tau_{i''}}):\u A_i\too \u B_i,\qquad i=1,\dots,p,
\end{equation}
where
     \[
\tau_{i^{(\varepsilon)}}:=\left\{
\begin{array}{ll}
1 & \hbox{if $\lambda _{i^{(\varepsilon)}}: t_i\longleftarrow  v_{i^{(\varepsilon)}}$}
             \\
-T & \hbox{if $\lambda _{i^{(\varepsilon)}}: t_i\longrightarrow  v_{i^{(\varepsilon)}}$}
                \end{array}
              \right.\quad
\text{for $i=1,\dots,p$ and $\varepsilon=0,1,2,$}
\]
in which $i^{(0)}:=i$, $i^{(1)}:=i'$, and $i^{(2)}:=i''$.

The array-representations \eqref{ccc} define the partitioned arrays
\[
\u A:=\u A_1\oplus\dots\oplus\u A_p,\qquad
\u B:=\u B_1\oplus\dots\oplus\u B_p.
\]
The sequence \eqref{yyy} defines the isomorphism
${\cal S}:{\cal A}\too {\cal B}$ if and only if
\eqref{mbh} holds if and only if
\begin{equation}\label{mbb}
(S_{v_1}^{\tau_1}\oplus\dots\oplus
S_{v_p}^{\tau_p},\,
S_{v_{1'}}^{\tau_{1'}}\oplus\dots\oplus
S_{v_{p'}}^{\tau_{p'}},\, S_{v_{1''}}^{\tau_{1''}}\oplus\dots\oplus
S_{v_{p''}}^{\tau_{p''}})
:\u A\too \u B.
\end{equation}

An arbitrary block-equivalence
\begin{equation}\label{aaq}
(R_1\oplus\dots\oplus
R_p,\,
R_{1'}\oplus\dots\oplus
R_{p'},\, R_{1''}\oplus\dots\oplus
R_{p''},\, )
:\u A\too \u B
\end{equation}
has the form \eqref{mbb} if and only if the following two conditions hold:
\begin{itemize}
  \item $R_{i^{(\varepsilon)}}=
R_{j^{(\delta)}}$ if $v_{i^{(\varepsilon)}}=
v_{j^{(\delta)}}$ and $\tau_{i^{(\varepsilon)}}=
\tau_{j^{(\delta)}}$,

  \item $R_{i^{(\varepsilon)}}=
R_{j^{(\delta)}}^{-T}$ if $v_{i^{(\varepsilon)}}=
v_{j^{(\delta)}}$ and $\tau_{i^{(\varepsilon)}}\ne
\tau_{j^{(\delta)}}$;
\end{itemize}
that is, if and only if \eqref{aaq} is a
linked block-equivalence with respect to the relations $\sim$ and $\Join$ on
\begin{equation*}\label{kjq}
{\cal P}:=\{1,\dots,p;1',\dots,p';1'',\dots,p''\}
\end{equation*}
that are defined as follows:
\begin{itemize}
  \item ${i^{(\varepsilon)}}\sim
{j^{(\delta)}}$ if $v_{i^{(\varepsilon)}}=
v_{j^{(\delta)}}$ and $\tau_{i^{(\varepsilon)}}=
\tau_{j^{(\delta)}}$,

  \item ${i^{(\varepsilon)}}\Join
{j^{(\delta)}}$ if $v_{i^{(\varepsilon)}}=
v_{j^{(\delta)}}$ and $\tau_{i^{(\varepsilon)}}\ne
\tau_{j^{(\delta)}}$.
\end{itemize}
Thus, ${\cal A}\simeq {\cal B}$ if and only if $\u A$ and $\u B$ are linked block-equivalent.

Let ${\cal X}=(\u X_1,\dots,\u
X_p)$ be a variable array-representation of $G$ of dimension $d=(d_1,\dots,d_q)$. Define the array $\u K({\cal X}):=\u X_1\oplus\dots\oplus
\u X_p$. Two array-representations $\cal
A$ and $\cal B$ of $G$ of dimension
$d$ are isomorphic if and only if
$\u K({\cal A})$ and $\u K({\cal B})$ are linked block-equivalent. This proves Theorem \ref{ktw1} due to Corollary \ref{ddr} and Theorem \ref{kts}.
\end{proof}

Note that the array $\u K({\cal X})$ in the proof of Theorem \ref{ktw1} is somewhat large. One can construct smaller arrays in most specific cases (as  in Example \ref{bbn}).

\section{Tensor-representations
of directed bipartite graphs}\label{hyu}

Definition \ref{aal} of representations of directed bipartite graphs is given in terms of arrays. In this section, we give an equivalent definition in terms of tensors that are considered as elements of tensor products, which may extend the range of validity of Theorem \ref{ktw1}.

Recall that a \emph{tensor on vector spaces}
$V_1,\dots, V_{m+n}$ over a field $\F$ (some of
them can be equal) is an element of the tensor product
\begin{equation}\label{kke}
{\cal T}\in V_1\otimes\dots\otimes  V_m \otimes  V_{m+1}^*\otimes\dots\otimes  V_{m+n}^*,
\end{equation}
in which $V_i^*$
denotes the dual space consisting of
all linear forms $V_i\to \F$. The
tensor ${\cal T}$ is called \emph{$m$ times
contravariant and $n$ times covariant}, or an \emph{$(m,n)$-tensor}
for short.

Choose a basis $f_{\alpha 1},\dots,f_{\alpha
d_{\alpha }}$ in each space $V_{\alpha
}$ (${\alpha}=1,\dots,m+n$) and take the
\emph{dual basis} $f_{{\alpha }1}^*,
\dots,f_{{\alpha }d_{\alpha}}^*$ in $V_{\alpha}^*$, where
$f_{{\alpha}i}^*:V_i\to \F$ is defined
by $f_{{\alpha}i}^*(f_{{\alpha}i}):=1$
and $f_{{\alpha}i}^*(f_{{\alpha}j}):=0$
if $i\ne j$. The tensor product in \eqref{kke} is the
vector space over $\F$ with the basis
formed by
\[
f_{1i_1}\otimes\dots\otimes f_{mi_m}\otimes f_{m+1,i_{m+1}}^*\otimes\dots\otimes
f_{m+n,i_{m+n}}^*,
\]
and so each tensor \eqref{kke} is
uniquely represented in the form
\begin{equation}\label{jyi}
{\cal T}=\sum_{i_1,\dots, i_{m+n}}a_{i_1\dots i_{m+n}}
f_{1i_1}\otimes\dots\otimes f_{mi_m}\otimes f_{m+1,i_{m+1}}^*\otimes\dots\otimes
f_{m+n,i_{m+n}}^*
\end{equation}
and can be given by \emph{its array $\u A=
[a_{i_1\dots i_{m+n}}]$ over $\mathbb
F$.}

Choose a new basis $g_{\alpha
1},\dots,g_{\alpha d_{\alpha }}$ in each $V_{\alpha }$, and let
$
S_{\alpha }:=[s_{\alpha ij}]$
be the change of basis matrix (i.e., $g_{\alpha j}=\sum_{i,j} s_{\alpha ij}f_{\alpha i}
$) for each
${\alpha}=1,\dots,m+n$. Denote by $\u B$ the array of $\cal T$ in the new set of bases.
Then
\begin{equation*}\label{jkl}
(S_1^{-T},\dots,S_m^{-T},S_{m+1}
,\dots,S_{m+n}):
\u A\too \u B,
\end{equation*}
which allows us to reformulate Definition \ref{aal} as follows.

\begin{definition}[see {\cite[Sec. 4]{ser_sur}}]
Let $G$ be a directed bipartite
graph  with ${\T}=\{t_1,\dots,t_p\}$
and $\V=\{1,\dots,q\}$.

\begin{itemize}
  \item A \emph{representation}
      $(\mathcal T,V)$ of $G$ is
      given by assigning  a finite
      dimensional vector space
      $V_v$ over $\F$ to each
      $v\in \V$ and  a tensor ${\cal
      T}_{\alpha }$ to each
      $t_{\alpha }\in \T$ so that if
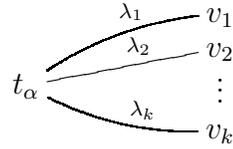
\begin{equation*}\label{liw}
\begin{split}
\xymatrix@R=-3pt@C=2cm
{
&v_1\\
&\\
&v_2\\
t_{\alpha }\ar@/^/@{-}[ruuu]^(0.6){\lambda_1}
\ar@{-}[ru]^(0.6){\lambda_2}
\ar@{}[r]
\ar@/_/@{-}[rdd]^(0.6)
{\lambda_k}
&\text{\raisebox{-3.5pt}{$\vdots$}}\\
&\\
&v_k\\
}
\end{split}
\end{equation*}
are all arrows in a vertex $t_{\alpha }$
(each line $\lin$ is
$\longrightarrow$ or
$\longleftarrow$; some of
$v_1,\dots,v_k\in \T$ may
coincide), then
\begin{equation*}\label{fei}
{\cal T}_{\alpha }\in V_{v_1}^{\varepsilon_1} \otimes\dots\otimes  V_{v_p}^{\varepsilon_p},\qquad
\varepsilon_i:=\left\{
                \begin{array}{ll}
                  1 & \hbox{if $\lambda _i: t_{\alpha} \longleftarrow  v_i$,} \\
                  * & \hbox{if $\lambda _i: t_{\alpha}\longrightarrow  v_i$.}
                \end{array}
              \right.
\end{equation*}

  \item Two representations
      $(\mathcal T,V)$ and
      $(\mathcal T',V')$ of $G$ are
      \emph{isomorphic} if there
      exists a system of linear
      bijections $\varphi_v:V_v\to
      V'_v$ ($v\in \V$) that together with
   the contragredient
      bijections
      $\varphi_v^{*-1}:V_v^*\to
      V_v^{\prime *}$ transform
      $\mathcal
      T_{1},\dots,\mathcal
      T_{p}$ to $\mathcal
      T'_{1},\dots,\mathcal
      T'_{p}$.
\end{itemize}
\end{definition}

\begin{example}
A special case of \eqref{jyi} is a $(0,2)$-tensor ${\cal T}=\sum_{i,j}a_{ij}
f_{1i}^*\otimes f_{2j}^*\in V_{1}^*\otimes V_{2}^*$. It can be identified with the
bilinear form
\[
{\cal T}:V_1\times V_2\to \F,\qquad (v_1,v_2)\mapsto\sum_{i,j}a_{ij}
f_{1i}^*(v_1)
f_{2j}^*(v_2).
\]
      A $(1,1)$-tensor
        ${\cal
        T}=\sum_{i,j}a_{ij}
        f_{1i}\otimes f_{2j}^*\in
        V_{1}\otimes V_{2}^*$ can
        be identified with the
        linear mapping
\[
{\cal T}:V_2\to V_1,\qquad v_2\mapsto\sum_{i,j}a_{ij}
f_{1i}
f_{2j}^*(v_2).
\]
The arrays $[a_{ij}]$ of these tensors are the matrices of these bilinear forms and linear mappings. The systems of tensors of order 2 are studied in \cite{hor-ser_form,hor-ser,ser_izv,ser_sur,ser_iso} as representations of graphs with undirected, directed, and double directed ($\longleftrightarrow$) edges that are assigned, respectively, by (0,2)-, (1,1)-, and (2,0)-tensors on the vector spaces assigned to the corresponding vertices. The problem of classifying such representations (i.e., of arbitrary systems of bilinear forms, linear mappings, and bilinear forms on dual spaces) was reduced in \cite{ser_izv} to the problem of classifying representations of quivers.
\end{example}

\begin{example}[see {\cite[Sec. 4]{ser_sur}}]\label{err}
Each representation
\begin{equation}\label{mmr}
\begin{split}
\xymatrix@C=95pt@R=5pt{
 {T_1}\ar@/^/@<0.3ex>[dr]&
  \\{T_2}\ar@/_/@<-0.3ex>[r]
  &{V_1}\ar@/_/@<-0.3ex>[l]
  \ar[l]\ar@/_/[ld]
 \\{T_3}\ar@/_/@<-0.3ex>[r]
 &{V_2}\ar[l]}
  \end{split}
\end{equation}
of the
bipartite directed graph \eqref{hle1}
consists of vector spaces $V_1$ and
$V_2$ and tensors
\begin{equation*}
T_1\in V_1,\qquad
T_2\in V_1^*\otimes
V_1^*\otimes
V_1,\qquad T_3\in
V_1^*\otimes
V_2^*\otimes V_2.
\end{equation*}
Each pair $(\Lambda, M)$, consisting of
a finite dimensional algebra $\Lambda$ with unit $1_{\Lambda}$
over a field $\mathbb F$ and a $\Lambda$-module $M$, defines the representation
\[
\xymatrix@C=95pt@R=5pt{
 {T_1=1_{\Lambda}}\ar@/^/@<0.3ex>[dr]&
  \\{T_2}\ar@/_/@<-0.3ex>[r]
  &{\Lambda_{\mathbb F}}\ar@/_/@<-0.3ex>[l]
  \ar[l]\ar@/_/[ld]
 \\{T_3}\ar@/_/@<-0.3ex>[r]
 &{M_{\mathbb F}}\ar[l]
  }
\]
of \eqref{hle1}, in which
\[
T_2=\sum_i a_{i1}^*\otimes
a_{i2}^*\otimes
a_{i3},\qquad
T_3=\sum_i a_{i}^*\otimes
m_{i1}^*\otimes m_{i2}
\]
(all $a_{ij}\in\Lambda$ and $m_{ij}\in M$) are the tensors that define the multiplications in $\Lambda$ and $M$:
\[
(a',a'')\mapsto
\sum_ia_{i1}^*(a')
a_{i2}^*(a'')
a_{i3},\qquad
(a,m)\mapsto
\sum_ia_{i1}^*(a)
m_{i1}^*(m)m_{i2}.
\]

Note that the identities (additivity,
distributivity, $\ldots$\:) that must satisfy the operations in $\Lambda$ and $M$ can
be written via tensor contractions in such a way that each representation \eqref{mmr} satisfying these relations defines a pair consisting of a finite dimensional algebra and its module. This leads to the theory of representations of bipartite directed graphs with relations that generalizes the theory of representations of quivers with relations.
\end{example}

\section*{Acknowledgements}

The authors are grateful to the referee  for the useful
suggestions.
V.~Futorny  was
supported by  CNPq grant 301320/2013-6 and by
FAPESP grant 2014/09310-5. J.A.~Grochow was supported by NSF grant DMS-1750319. The work
of V.V.~Sergeichuk was supported by
FAPESP grant 2015/05864-9 and was done during his
visits to the Santa Fe Institute (the
workshop ``Wildness in computer
science, physics, and mathematics'') and to the University of
S\~ao Paulo.

\end{document}